\documentclass[10pt]{article}
\usepackage{amssymb}
\usepackage{graphicx}
\usepackage{xcolor} 
\usepackage{tensor}
\usepackage{fullpage} 
\usepackage{amsmath}
\usepackage{amsthm}
\usepackage{verbatim}
\usepackage{hyperref}
\usepackage{enumitem}
\usepackage{mathrsfs}
\setlist[enumerate]{leftmargin=1.5em}
\setlist[itemize]{leftmargin=1.5em}

\providecommand{\bysame}{\leavevmode\hbox to3em{\hrulefill}\thinspace}
\providecommand{\MR}{\relax\ifhmode\unskip\space\fi MR }

\providecommand{\href}[2]{#2}

\usepackage{seqsplit} 
\definecolor{green}{rgb}{0,0.8,0} 

\newtheorem{theorem}{Theorem}[section]

\newtheorem{lemma}[theorem]{Lemma}
\newtheorem{proposition}[theorem]{Proposition}

\theoremstyle{definition}

\theoremstyle{remark}
\newtheorem{remark}[theorem]{Remark}

\numberwithin{equation}{section}
\newcommand{\nrm}[1]{\Vert#1\Vert}

\newcommand{\nnrm}[1]{{\vert\kern-0.25ex\vert\kern-0.25ex\vert #1 
    \vert\kern-0.25ex\vert\kern-0.25ex\vert}}

\newcommand{\supp}{{\mathrm{supp}}\,}

\newcommand{\ud}{\mathrm{d}}
\newcommand{\rd}{\partial}
\newcommand{\nb}{\nabla}

\newcommand{\alp}{\alpha}

\newcommand{\dlt}{\delta}

\newcommand{\omg}{\omega}

\newcommand{\bfu}{{\bf u}}

\newcommand{\bbN}{\mathbb N}

\newcommand{\bbR}{\mathbb R}

\begin{document}

\title{A simple ill-posedness proof for incompressible Euler equations \\ in critical Sobolev spaces}
\author{Junha Kim\thanks{School of Mathematics, Korea Institute for Advanced Study, 85 Hoegi-ro, Dongdaemun-gu, Seoul 02455. E-mail: junha02@kias.re.kr} \and In-Jee Jeong\thanks{Department of Mathematics and RIM, Seoul National University, 1 Gwanak-ro, Gwanak-gu, Seoul 08826, Korea. E-mail: injee\_j@snu.ac.kr, Corresponding Author.}} 
\date{\today}



\maketitle
\renewcommand{\thefootnote}{\fnsymbol{footnote}}
\footnotetext{\emph{Key words:} incompressible Euler; ill-posedness; vortex stretching; critical regularity; non-existence; norm inflation.\\ \quad\emph{2020 AMS Mathematics Subject Classification:} 76B47, 35Q35 }

\begin{abstract}
	We provide a simple proof that the Cauchy problem for the incompressible Euler equations in $\mathbb{R}^{d}$ with any $d\ge3$ is ill-posed in critical Sobolev spaces, extending an earlier work of Bourgain--Li (Int. Math. Res. Not. 2021) in the case $d = 3$. The ill-posedness is shown for certain critical Lorentz spaces as well. 
\end{abstract}


\section{Introduction}

\subsection{Main results}

In this paper, we are concerned with the Cauchy problem for the incompressible Euler equations in $\bbR^d$ with $d \geq 3$:
\begin{equation}  \label{eq:euler}
\left\{
\begin{aligned} 
\rd_t \bfu + \bfu \cdot \nb \bfu + \nb p &= 0, \\
\nb \cdot \bfu &= 0, \\
\bfu(t = 0) &= \bfu_0.
\end{aligned}
\right.
\end{equation} We shall provide a simple proof that \eqref{eq:euler} is \textit{ill-posed} in the critical Sobolev space $H^{d/2+1}(\bbR^d)$, in the strongest sense of Hadamard. The space $H^{d/2+1}(\bbR^d)$ is critical since one can obtain local well-posedness of \eqref{eq:euler} in $H^s(\bbR^d)$ with any $s>d/2+1$ (\cite{Ka,KL}), while $\bfu \in H^{d/2+1}(\bbR^d)$ does not guarantee that $\nrm{\nb\bfu}_{L^{\infty}} < \infty$, which is an essential ingredient in local regularity proof. Our first main result gives the existence of a $H^{d/2+1}$--regular data without any solutions in the same space. 
\begin{theorem}[Nonexistence]\label{thm:nonexist}
	For any $\varepsilon> 0$, there exists $\bfu_{0} \in H^{d/2+1}(\bbR^d)$ satisfying \begin{equation*}
		\begin{split}
			\nrm{\bfu_0}_{H^{d/2+1}(\bbR^d)} < \varepsilon 
		\end{split}
	\end{equation*} such that there is \emph{no} corresponding solution to \eqref{eq:euler} belonging to $L^\infty([0,\dlt];H^{d/2+1}(\bbR^{d}))$ for \emph{any} $\dlt>0$. 
\end{theorem}

Furthermore, we are able to prove that a large growth of the $H^{d/2+1}$--norm occurs for certain infinitely smooth initial data.  
\begin{theorem}[Strong norm inflation]\label{thm:inf}
	For any $\delta$, $\varepsilon$, $A > 0$, there exists $\bfu_0 \in C^\infty(\bbR^d)$ with $\nrm{\bfu_0}_{H^{d/2+1}(\bbR^d)} < \varepsilon$ such that the unique local-in-time smooth solution $\bfu$ to \eqref{eq:euler} exists on $[0,\delta^*]$ for some $0 < \delta^* < \delta$ and satisfies
	\begin{equation*}
		\begin{split}
			\sup_{t \in [0,\delta^*]} \nrm{\bfu(t,\cdot)}_{H^{d/2+1}(\bbR^d)} > A .
		\end{split}
	\end{equation*}
\end{theorem}

We remark that the above statements for the two and three-dimensional cases have been already proved by Bourgain--Li in \cite{BL1} and \cite{BL3D}, respectively. We refer the interested readers to these works for an extensive list of literature on the Cauchy problem for the Euler equations. As Bourgain and Li are emphasizing in the introduction of \cite{BL3D}, the 3D case is much more involved than the 2D case due to the presence of the vortex stretching term, lack of $L^p$ conservation of the vorticity, absence of Yudovich theory, and the issues arising from dealing with the non-local norm $H^{5/2}$. A short proof for the 2D case was given in \cite{EJ}, and this simplified argument became a basis for further developments in the well-posedness theory, including continuous-in-time loss of Sobolev regularity (\cite{Jeong}), enstrophy growth in the Navier--Stokes equations (\cite{JY,JY2}), and strong ill-posedness for active scalar equations (\cite{JKim}). Some further developments in this direction are given in \cite{Kwon,Co}.  Therefore, we believe that the simplified argument presented in the current paper for three and higher dimensions will be useful in the further study of well-posedness for the high-dimensional Euler and related systems. Indeed, the ill-posed mechanism in critical spaces is closely related with recent groundbreaking works on finite-time singularity formation for the 3D Euler equations (\cite{Elgindi-3D,ChenHou}). Moreover, our simple proof is robust enough to handle all dimensions $d \ge 3$ at once. Lastly, we prove strong ill-posedness in critial Lorentz spaces which does not seem to follow from the approach of \cite{BL3D}, see technical discussion in \S \ref{subsec:ideas} below. 

\subsection{Axisymmetric Euler without swirl} The aforementioned ill-posedness results will be obtained by considering a class of solutions referred to as \textit{axisymmetric without swirl}: if the velocity takes the form \begin{equation}\label{eq:axisym-no-swirl}
	\begin{split}
		\bfu(t,x) = u^r(t,r,x_d) e^r + u^d(t,r,x_d) e^d 
	\end{split}
\end{equation} where $x = (x_h, x_d)$ with $r = |x_h|$, then it can be shown that this ansatz propagates in time for the evolution of \eqref{eq:euler}, upon having the uniqueness of a solution to \eqref{eq:euler}. In this case, introducing the scalar vorticity \begin{equation}\label{eq:vort-def}
\begin{split}
	\omg(t,x) = \rd_r u^d - \rd_d u^r 
\end{split}
\end{equation} which is again a function of $r$ and $x_d$, the Euler equations \eqref{eq:euler} reduce to a simple evolution equation for $\omg$: we have \begin{equation}\label{eq:Euler-axisym-no-swirl}
\begin{split}
	\rd_t \frac{\omg}{r^{d-2}} + (u^r \rd_r + u^d \rd_d) \frac{\omg}{r^{d-2}} = 0. 
\end{split}
\end{equation} The components $u^r$ and $u^d$ can be recovered from $\omg$ at each instant of time via the axisymmetric Biot--Savart law: \begin{equation}\label{eq:BS-r}
\begin{split}
	u^r(t,x) =  -\frac{1}{d|B_d|}  \frac{1}{|x_h|}\int_{\bbR^d} \frac{ (x_d-y_d) \sum_{j=1}^{d-1} x_j y_j }{|y_h||x-y|^d}  \omg(t,y) \, \ud y, 
\end{split}
\end{equation} \begin{equation}\label{eq:BS-d}
\begin{split}
	u^d (t,x) = \frac{1}{d|B_d|} \int_{\bbR^d} \frac{  \sum_{j=1}^{d-1} (x_j-y_j) y_j }{|y_h||x-y|^d}  \omg(t,y) \, \ud y,
\end{split}
\end{equation} where $|B_d|$ denotes the volume of the $d$-dimensional unit ball. 
Introducing the flow $\Phi = \Phi^r e^r + \Phi^d e^d$ by 
\begin{equation}\label{eq:flow}
	\begin{split}
		\Phi(0,x) = x \qquad \mbox{and} \qquad \frac {\mathrm{d}}{\mathrm{d}t} \Phi(t,x) = \bfu(t,\Phi(t,x)),
	\end{split}
\end{equation}
we see that solutions to \eqref{eq:Euler-axisym-no-swirl} satisfy 
\begin{equation}\label{eq:cauchy}
	\begin{split}
		\frac {\omega(t,\Phi(t,x))}{\Phi^r (t,x)^{d-2}} = \frac {\omega_0(x)}{r^{d-2}}
	\end{split}
\end{equation} along the flow. Note that along the flow, $\omg$ could either increase or decrease by the factor $(\Phi^r/r)^{d-2}$. This \textit{vortex stretching} effect can be alternatively seen from the equation for $\omg$, rather than $\omg/r^{d-2}$: \begin{equation*}
\begin{split}
		\rd_t  {\omg}  + (u^r \rd_r + u^d \rd_d) {\omg} = (d-2) \frac{u^r}{r}\omg . 
\end{split}
\end{equation*}

Regarding the axisymmetric without swirl equation \eqref{eq:Euler-axisym-no-swirl}, we have local well-posedness for any initial data satisfying $\omg_{0} \in L^1 \cap L^\infty(\bbR^d)$ and $r^{-(d-2)}\omg_{0} \in L^{d,1}(\bbR^d)$: there exists a unique corresponding local-in-time solution to \eqref{eq:euler} satisfying $\omg \in L^1 \cap L^\infty(\bbR^d)$ and $r^{-(d-2)}\omg \in L^{d,1}(\bbR^d)$ for some time interval. The solution is global in the case $d = 3$. This was proved by Danchin (\cite{Danchin}) and can be viewed as a sharp extension of the celebrated Yudovich theory \cite{Yudovich1963} in two dimensions. Unlike the two-dimensional case, the extra condition $r^{-(d-2)}\omg_{0} \in L^{d,1}(\bbR^d)$ is necessary to control vortex stretching, and seems to be \textit{sharp} (see \cite{Elgindi-3D}). Indeed, as a byproduct of our analysis, we were able to obtain the following result: 

\begin{theorem}\label{thm:Lorenz} For each $d\ge3$, there exists $1< \bar{q}_d < \infty$ such that the axisymmetric equation \eqref{eq:Euler-axisym-no-swirl} is \emph{ill-posed} for $\omg_{0} \in L^1 \cap L^\infty$ with $r^{-(d-2)}\omg_{0}\in L^{d,q}(\bbR^d)$ with any $\infty\ge q > \bar{q}_d$. 
\end{theorem} The precise ill-posedness statement is given in Proposition \ref{prop:illposed} below. 

\subsection*{Organization of the paper} The rest of the paper is organized as follows. In \S \ref{sec:prelim}, we explicitly show the choice of initial data used to prove the main theorems. After stating the Key Lemma which is the main technical tool, we explain the ideas of the proof. Then, in \S \ref{sec_std}, we make several key observations on dynamics of the solution associated with the initial data chosen in \S \ref{sec:prelim}. After that, Theorems \ref{thm:Lorenz}, \ref{thm:nonexist}, \ref{thm:inf} are proved in \S \ref{sec:Linfty}, \ref{sec:nonexistence}, \ref{sec:norm-inflation}, respectively. Finally, we prove the Key Lemma in \S \ref{sec:key-Lemma}.

\subsection*{Notation} We denote $B(x;r)$ be the open ball in $\bbR^d$ of radius $r$ centered at $x$. The $d$-dimensional volume of the unit ball in $\bbR^d$ is denoted by $|B_d|$. As it is usual, we write $A \lesssim B$ if there exists an absolute constant $C>0$ such that $A \le CB$. Furthermore, we say $A \simeq B$ if $A\lesssim B$ and $B \lesssim A$. 

\section{Preliminaries} \label{sec:prelim}

\subsection{Choice of initial data} We demonstrate the choice of initial data that will be used in the proof of main theorems.  To begin with, let $\phi: \bbR^2 \rightarrow \bbR_{\ge0}$ be a smooth bump function with the following properties: 
\begin{itemize}
	\item $\phi$ is $C^\infty$-smooth and radial;
	\item $\phi$ is supported in $B(0;\frac{1}{8})$ and $\phi = 1$ in $B(0;\frac{1}{32})$. 
\end{itemize} For some $n_0 < m \leq \infty$ and $0 < \alpha < 3/4$ to be determined later, we take
\begin{equation}\label{eq:nonexist-data}
	\begin{split}
		\omega_0 :=  \sum_{n=n_0}^{m}  n^{-\alp} \omega^{(n)}_{0, loc},
	\end{split}
\end{equation}
where
\begin{equation}\label{eq:nonexist-bubble}
	\begin{split}
		\omega^{(n)}_{0, loc}(r,x_d) := \phi( 8^n(r-8^{-n+1}, x_d-8^{-n} ) ) - \phi( 8^n(r-8^{-n+1}, x_d+8^{-n} ) ).
	\end{split}
\end{equation}
The precise value of $\alp$ (which determines the regularity of \eqref{eq:nonexist-data}) will be specified in the proofs of each theorems. We observe that $\omg_0$ is an odd function with respect to $x_d$, and has the form of a weighted sum of normalized \textit{bubbles}; we shall sometimes (informally) refer to $n^{-\alp} \omega^{(n)}_{0, loc}$ as the $n$-th bubble.

\subsection{The Key Lemma}\label{sec:key} We now state the main technical tool of this paper. 
\begin{lemma}\label{key_lem}
	We impose the following assumptions on $\omega \in H^{\frac d2} \cup (L^{\infty} \cap L^2) (\bbR^d)$:
	\begin{itemize}
		\item Odd with respect to the last coordinate: 
		\begin{equation}\label{cond_d_sym}
			\begin{split}
				\omega(y_h,y_d) = -\omega(y_h,-y_d);
			\end{split}
		\end{equation}
		\item For any $y \ne 0$ satisfying either $|y_h| = 0$ or $y_d=0$, there exists an open neighborhood of $y$ such that $\omega$ vanishes. 
	\end{itemize} Then, there exists a constant $C>0$ such that for any $x \in \bbR^d$ with $r=|x_h| \geq x_d > 0$, we have 
	\begin{equation}\label{ur_est}
		\begin{split}
			\left| \frac {u^r(x)}{r} -  {\frac 1{(d-1)|B_d|}} \int_{Q(x)} \frac { |y_h| y_d }{|y|^{d+2}} \omega(y) \,\mathrm{d}y \right| \le C B_1(x)
		\end{split}
	\end{equation}
	and
	\begin{equation}\label{ud_est}
		\begin{split}
			\left| \frac {u^d(x)}{x_d} + {\frac 1{|B_d|}} \int_{Q(x)} \frac { |y_h| y_d }{|y|^{d+2}} \omega(y) \,\mathrm{d}y \right| \le C B_2(x) ,
		\end{split}
	\end{equation}
	where $Q(x) := \{ y \in \bbR^d ; |y_h| \geq 4|x_h| \}$ and $B_1$, $B_2$ are non-negative functions with the upper bound
	\begin{equation*}
		\begin{split}
			B_1(x) \leq \min \left\{ \| \nabla \omega \|_{L^d(\bbR^d)}, \| \omega \|_{L^{\infty}(\bbR^d)} \right\},
		\end{split}
	\end{equation*}
	\begin{equation*}
		\begin{split}
			B_2(x) \leq \min \left\{ \left(1 + \log \frac {r}{x_d} \right)^{\frac {d-1}d} \| \nabla \omega \|_{L^d(\bbR^d)}, \left(1 + \log \frac {r}{x_d} \right) \| \omega \|_{L^{\infty}(\bbR^d)} \right\}.
		\end{split}
	\end{equation*}
\end{lemma}
\begin{remark}
	Gagliardo--Nirenberg interpolation inequality implies that
	\begin{equation*}
		\begin{split}
			\| \nabla \omega \|_{L^d(\bbR^d)} \leq C\| \Lambda^{\frac d2} \omega \|_{L^2(\bbR^d)}.
		\end{split}
	\end{equation*}	
\end{remark} 

The ``Key Lemma'' has first appeared in a celebrated work of Kiselev--Sverak \cite{KS} for the two-dimensional Euler equations, although Bourgain and Li were using similar formulas in \cite{BL1}; see also \cite{Z,KRYZ,JY}. For the three-dimensional Euler equations, a similar result with a different remainder bound was used in \cite{Elgindi-3D,ChenHou}. Our version of Key Lemma is sharp in that the remainder is bounded using only critical norms of the Euler equations. Therefore, we believe that this lemma will be useful for other purposes, for instance in the construction of smooth solutions with rapid Sobolev norm growth. While we defer the (somewhat tedious) proof later to \S \ref{sec:key-Lemma}, one may observe from the axi-symmetric Biot--Savart law that in the limit $x\to 0$, we exactly have \begin{equation*}
	\begin{split}
		\rd_r u^r (0) = \lim_{x \to 0} \frac {u^r(x)}{r} =  {\frac 1{(d-1)|B_d|}} \int_{ \bbR^d } \frac { |y_h| y_d }{|y|^{d+2}} \omega(y) \,\mathrm{d}y, \quad \rd_d u^d (0) = \lim_{x \to 0} \frac {u^d(x)}{x_d} = - {\frac 1{|B_d|}} \int_{\bbR^3} \frac { |y_h| y_d }{|y|^{d+2}} \omega(y) \,\mathrm{d}y . 
	\end{split}
\end{equation*}

\subsection{Hardy's inequality}

We now recall the famous Hardy's inequality. We omit the proof since it can be easily obtained using an integration by parts. 
\begin{lemma}[Hardy's inequality]\label{lem_hd}
	Let $f$ be a positive smooth function defined on the interval $(0,1)$ with $f(0) = 0$. Then for any $p>1$, we have
	\begin{equation*}
		\begin{split}
			\nrm{x^{-1}f(x)}_{L^p(0,1)} \leq C \nrm{\partial_{x} f(x)}_{L^p(0,1)}.
		\end{split}
	\end{equation*}
\end{lemma} 

\subsection{Ideas of the proof}\label{subsec:ideas}

Let us illustrate the main ideas of the proof, which have been inspired by the previous works \cite{EJ,JY,JKim}. The interested reader can find more information there; in the current setup, the main difference is the presence of vortex stretching. We consider the case of $\bbR^3$. 
\begin{itemize}
	\item \textit{Formal computation and mechanism of illposedness.} With a few simple and formal computations, let us explain the reasoning behind the choice of our initial data. To begin with, at the initial time, the velocity is roughly given by \begin{equation*}
		\begin{split}
			\begin{pmatrix}
				u^r_0 \\ u^z_0
			\end{pmatrix} = \begin{pmatrix}
			r \\ -2z 
		\end{pmatrix} (\ln\frac{1}{|x|})^{1-\alp} \varphi(\frac{x}{|x|}), 
		\end{split}
	\end{equation*} when $n_0=1$ and $m=\infty$ in \eqref{eq:nonexist-data}, which can be seen by the main term of the Key Lemma. Here, $\varphi$ is a bounded function supported in the region $r, |z|\simeq |x|$. Then, using that the variable $\xi:=\omg/r$ is simply being transported by the flow,  we compute\begin{equation*}
		\begin{split}
			\frac{d}{dt}|_{t=0} \int_{\bbR^3} \frac{\xi^2}{z} \, \ud x = \int_{\bbR^3} \frac{-u^z_0}{z} \frac{\xi_0^2}{z} \, \ud x   \simeq \int_{\{|x| \le 1\} \cap \bbR^{3}}  (\ln \frac{1}{|x|})^{1-3\alp} \frac{1}{|x|^3} \, \ud x  =   +\infty
		\end{split}
	\end{equation*} when $\alp<2/3$.  That is, the quantity $\nrm{z^{-\frac12}r^{-1}\omg}_{L^2}$ immediately blows up for $t>0$, which implies the blow up of the critical norm $\nrm{\omg}_{H^{3/2}}$ as it controls $\nrm{z^{-\frac12}r^{-1}\omg}_{L^2}$ thanks to the Hardy's inequality. Indeed, when $\alp$ is slightly larger than $1/2$, the initial vorticity ``barely'' belongs to the space $H^{3/2}$ in a way that the corresponding velocity is not Lipschitz continuous in space. Geometrically, the ill-posedness comes from the squeezing of the vorticity in the $z$-direction by the strong hyperbolic flow. In the case of $L^\infty$, we can similarly compute \begin{equation*}
		\begin{split}
			\frac{d}{dt}|_{t=0} \nrm{\omg}_{L^\infty} = \nrm{ r^{-1}u_{0}^{r} \omg_0 }_{L^\infty} \simeq \sup_{|x|\le1} (\ln\frac{1}{|x|})^{1-2\alp} = +\infty 
		\end{split}
	\end{equation*} when $\alp<1/2$, and one clearly sees that the ill-posedness comes from vortex stretching. 
	\item \textit{Monotonicity and stability.} While the above computation shows that ill-posedness occurs for the linear advection equation obtained by replacing $u(t)$ by $u_0$, for the nonlinear problem we need to understand the effects on the velocity caused by evolution of the vorticity, and in particular need to make sure that the velocity gradient (given by the main terms in the Key Lemma) remains large enough to create norm inflation. There are two competing effects; while the squeezing of the vorticity weakens the induced velocity gradient, the vortex stretching term strengthens it. As in our previous work \cite{JKim}, we identify a timescale for each of the bubbles during which it essentially remains the same, and obtain a lower bound on the velocity gradient based on such a stability statement. It turns out that, \textit{together with the help of the vortex stretching term}, such a lower bound is sufficient to prove strong ill-posedness results. 
	\item \textit{Contradiction hypothesis and Lagrangian property.} Throughout the proof, we heavily rely on the Lagrangian analysis, and in particular on the fact that the solutions of the axisymmetric Euler equations satisfy the Cauchy formula \eqref{eq:cauchy}. The Lagrangian analysis is not only essential in proving the stability statement described in the above but also allows us to obtain sharp bounds on the critical norms. Fortunately, existence and uniqueness of the flow map in our case is guaranteed by the contradiction hypothesis that there is a local solution in the critical space. 
\end{itemize}

\begin{remark}\label{rmk:BL}
	At this point, let us point out the main difference of our argument with the original proof of Bourgain--Li: in \cite{BL3D}, to prove ill-posedness in $H^{3/2}(\bbR^3)$,  the authors very carefully identify a class of initial data such that there is local well-posedness in $L^\infty(\bbR^3)$ (yet escapes $H^{3/2}$). In their work, a significant portion is devoted to proving propagation of the $L^\infty$--norm with such data. After having $L^\infty$, one can then try to repeat the strategy developed in the 2D case. On the other hand, we proceed by a direct contradiction argument which does not require the hypothetical solution to be bounded in $L^\infty$. Furthermore, we are able to prove ill-posedness in $L^\infty$ as well: the vortex stretching effect is crucial in Theorem \ref{thm:Lorenz}.
\end{remark}

\begin{remark}
	Very recently, An--Chen--Yin \cite{com1,com2} proved strong ill-posedness of the \textit{compressible} Euler equations in critical Sobolev spaces for dimensions 2 and 3. The ill-posedness in the compressible Euler case seems to be driven by instantaneous shock formation.  
\end{remark}

\section{Short time dynamics}\label{sec_std}
We consider $\omega_{0}$ of the form \eqref{eq:nonexist-data} and proceed under the assumption that there exists an associated solution to \eqref{eq:Euler-axisym-no-swirl}
\begin{equation}\label{omg_class}
	\begin{split}
		\omega \in L^{\infty}([0,T];H^{\frac d2} \cup (L^{\infty} \cap L^2) (\bbR^d))
	\end{split}
\end{equation}
with \begin{equation}\label{eq:sol-bound}
	\begin{split}
		\sup_{t \in [0,T]} \min \left\{ \| \Lambda^{\frac d2} \omega(t) \|_{L^2 (\bbR^d)}, \| \omega(t) \|_{L^{\infty}(\bbR^d)} \right\} \leq A
	\end{split}
\end{equation} for some $T>0$ and $A>0$. By this assumption, we can find a constant $B>0$ such that
\begin{equation*}
	\begin{split}
		\| \nabla u(t) \|_{BMO(\bbR^d)} + \| \omega(t) \|_{L^2(\bbR^d)} \leq B, \qquad t \in [0,T].
	\end{split}
\end{equation*} This follows from the embedding $H^{d/2} \subset BMO$ and the singular integral operator bound $L^\infty \to BMO$. Moreover, the $BMO$ bound  implies that the equation \eqref{eq:Euler-axisym-no-swirl} possesses at most one solution belonging to the class \eqref{omg_class} and the velocity $u(t,\cdot)$ is log-Lipschitz (see \cite{JKim,Bed}): there exists a constant $C>0$ such that
\begin{equation*}
	\begin{split}
		\sup_{t \in [0,T]} |u(t,x) - u(t,y)| \leq C B |x-y| \log \left( 10 + \frac 1{|x-y|} \right), \qquad x,y \in \bbR^d.
	\end{split}
\end{equation*} This guarantees that the flow map $\Phi$ is well-defined on $t \in [0,T]$ as the unique solution of \eqref{eq:flow}. Thus, we can verfiy that the unique solution $\omega$ is given by \eqref{eq:cauchy} and in particular Lemma~\ref{key_lem} applies to $\omg(t,\cdot)$ on the time interval $[0,T]$. We shall let $T$ be small so that $AT \leq c_0$ for some absolute constant $c_0>0$, which is determined at the end of this section.  We begin with the simple result.
\begin{lemma}\label{lem:basic}
	Assume that $\omega$ be a solution to \eqref{eq:Euler-axisym-no-swirl} satisfying \eqref{eq:sol-bound} with initial data \eqref{eq:nonexist-data}. Then, we have $\omega(t,x) = 0$ for all $x \in \bbR^d$ with $0 \leq |x_h| < |x_d|$ and $t \in [0,T]$.
\end{lemma} \begin{proof}
	We recall Lemma~\ref{key_lem} with $x$ on the region $\{(r,x_d) \in \bbR_{>0} \times \bbR_{>0};1/2 \leq x_d/r \leq1\}$
	\begin{equation*}
		\begin{split}
			\frac {u^r(t,x)}{r} \geq -CA, \qquad \frac {u^d(t,x)}{x_d} \leq CA.
		\end{split}
	\end{equation*}
	Using these inequalities, we show that $\Phi^d(t,x)/\Phi^r(t,x) \leq 1$ for all $x$ with $x_d/r \leq 1/2$ and $t \in [0,T]$. We don't need to consider $x$ with $x_d/r > 1/2$ due to the definition of $\omega_0$. Suppose that $\Phi(t,x)$ enters the region at $t = T^*\leq T$. Until escaping the area, the trajectory complies with
	\begin{equation*}
		\begin{split}
			\frac {\ud}{\ud t} \log \frac {\Phi^d(t,x)}{\Phi^r(t,x)} \leq CA.
		\end{split}
	\end{equation*}
	Then for $CAT \leq C c_0 \leq \log 2$, it follows
	\begin{equation*}
		\begin{split}
			\log \frac {\Phi^d(t,x)}{\Phi^r(t,x)} \leq \log \frac 12 + CAT \leq 0.
		\end{split}
	\end{equation*}
	From this, we conclude that $\Phi(t,x)$ cannot touch the line $r=x_d$. This completes the proof.
\end{proof} 

This lemma guarantees that $\Phi(t,x)$ with $x \in \supp (\omega_0)$ can be traced by Lemma~\ref{key_lem} up to $t \in [0,T]$. For simplicity, let us write $n$-th bubble as $\Omega_n := \supp (\omega^{(n)}_{0, loc})$. The next result provides that the bubbles are ``well-ordered": each size of bubble and distance of adjacent bubbles in direction $r$ are maintained to some extent.
\begin{lemma}\label{lem_order}
	Assume that $\omega$ be a solution to \eqref{eq:Euler-axisym-no-swirl} satisfying \eqref{eq:sol-bound} with initial data \eqref{eq:nonexist-data}. Then, we have
	\begin{equation}\label{eq:claim}
		\begin{split} 
			\sup_{x \in \Omega_n} \Phi^r(t,x) \leq 4\inf_{x \in \Omega_n} \Phi^r(t,x) \qquad \mbox{and} \qquad 4 \sup_{x \in \Omega_{n+1}} \Phi^r(t,x) \leq \inf_{x \in \Omega_n} \Phi^r(t,x)		
		\end{split}
	\end{equation} for all $n \geq n_0$ and $t \in [0,T]$. 
\end{lemma}
\begin{proof}
	We first show a generic proerty, which works even for $\omega_0$ without \eqref{eq:nonexist-data}: let $x,x' \in \bbR^{d-1}\times \bbR_{>0}$ with $r \geq r'>0$. Then, we have
	\begin{equation}\label{gen_prop}
		\begin{split} 
			\frac 12 \Phi^r(t,x') \leq \Phi^r(t,x) \leq \frac {2r}{r'} \Phi^r(t,x'), \qquad t \in [0,T].
		\end{split}
	\end{equation}
	To show it, we make a simple observation
	\begin{equation*}
		\begin{split} 
			\frac {u^r(t,x)}{r} - \frac {u^r(t,x')}{r'} \le CA, \qquad r \geq r' > 0.
		\end{split}
	\end{equation*}
	This is directly obtained from \eqref{ur_est} with $y_d \omega(y) \geq 0$. Then, we have
	\begin{equation*}
		\begin{split} 
			\frac {\ud}{\ud t} \log \Phi^r(t,x) - \frac {\ud}{\ud t} \log \Phi^r(t,x') \le CA
		\end{split}
	\end{equation*}
	when $\Phi^r(t,x) \geq \Phi^r(t,x')$. Integrating it over time and using the smallness of $c_0$, we can infer the claim. 
	
	Now, we consider $\omega_0$ with \eqref{eq:nonexist-data}. Then, we clearly have from \eqref{gen_prop} that
	\begin{equation}\label{order_est}
		\begin{split} 
			\sup_{x \in \Omega_{n'}} \Phi^r(t,x) \leq 4\inf_{x \in \Omega_n} \Phi^r(t,x), \qquad t \in [0,T]
		\end{split}
	\end{equation}
	for all $n' \geq n \geq n_0$. Thus, it remains to show that
	\begin{equation}\label{order_est2}
		\begin{split} 
			4 \sup_{x \in \Omega_{n+1}} \Phi^r(t,x) \leq \inf_{x \in \Omega_n} \Phi^r(t,x), \qquad t \in [0,T]
		\end{split}
	\end{equation}
	for all $n \geq n_0$. We prove it inductively. Let $x \in \Omega_{n_0}$ and $x' \in \Omega_{n_0+1}$ with $x_d,x_d'>0$. We have from \eqref{ur_est} that
	\begin{equation*}
		\begin{split} 
			\frac {\ud}{\ud t} \log \Phi^r(t,x') - \frac {\ud}{\ud t} \log \Phi^r(t,x) \leq \frac 1{(d-1)|B_d|} \int_{Q(\Phi(t,x')) \setminus Q(\Phi(t,x))} \frac { |y_h| y_d }{|y|^{d+2}} \omega(t,y) \,\mathrm{d}y + CA.
		\end{split}
	\end{equation*}
	We notice from \eqref{order_est} that
	\begin{equation*}
		\begin{split} 
			\int_{Q(\Phi(t,x')) \setminus Q(\Phi(t,x))} \frac { |y_h| y_d }{|y|^{d+2}} \omega(t,y) \,\mathrm{d}y \leq \int_{\Phi(t,\Omega_{n_0})} \frac { |y_h| |y_d| }{|y|^{d+2}} |\omega(t,y)| \,\mathrm{d}y .
		\end{split}
	\end{equation*}
	Here, we consider $n \geq n_0$ for a while to give an estimate uniform on $n$. By the volume preserving property of the flow, we can see
	\begin{equation}\label{In_est}
		\begin{split} 
			\int_{\Phi(t,\Omega_{n})} \frac { |y_h||y_d| }{|y|^{d+2}} |\omega(t,y)| \,\mathrm{d}y = \int_{\Phi(t,\Omega_{n})} \frac {|y_h|^{d-1} |y_d|}{|y|^{d+2}} \frac {|\omega(t,y)|}{|y_h|^{d-2}} \,\mathrm{d}y = \int_{\Omega_{n}} \frac {|\Phi^r(t,x)|^{d-1}|\Phi^d(t,x)|}{|\Phi(t,x)|^{d+2}} \frac {|\omega_0(y)|}{|y_h|^{d-2}} \,\mathrm{d}y.
		\end{split}
	\end{equation}
	Thus, we have
	\begin{equation*}
		\begin{split} 
			\int_{\Phi(t,\Omega_{n})} \frac { |y_h||y_d| }{|y|^{d+2}} |\omega(t,y)| \,\mathrm{d}y \leq \left( \sup_{x \in \Omega_{n}} \frac {r}{\Phi^r(t,x)} \right)^2 \int_{\Omega_{n}} \frac {|\omega_0(y)|}{|y_h|^{d}} \,\mathrm{d}y.
		\end{split}
	\end{equation*}
	Since \eqref{ur_est} implies
	\begin{equation*}
		\begin{split} 
			\frac {u^r(t,x)}{r} \ge -CA,
		\end{split}
	\end{equation*}
	we have
	\begin{equation*}
		\begin{split} 
			\frac {\Phi^r(t,x)}{r} \geq \exp \left( -CAt \right).
		\end{split}
	\end{equation*}
	These yield
	\begin{equation*}
		\begin{split} 
			\int_{\Phi(t,\Omega_{n})} \frac { |y_h||y_d| }{|y|^{d+2}} |\omega(t,y)| \,\mathrm{d}y \leq C n^{-\alpha} \exp \left( CAt \right), \qquad n \geq n_0.
		\end{split}
	\end{equation*}
	Applying this estimate, we deduce
	\begin{equation*}
		\begin{split} 
			\frac {\ud}{\ud t} \log \Phi^r(t,x') - \frac {\ud}{\ud t} \log \Phi^r(t,x) \leq C \exp \left( CAt \right) + CA.
		\end{split}
	\end{equation*}
	Taking $c_0>0$ small enough, we can infer
	\begin{equation*}
		\begin{split} 
			\log \Phi^r(t,x') - \log \Phi^r(t,x) \leq \frac {\exp \left( CAT \right)-1}{A} + CAT + \log r' - \log r \leq - \log 4.
		\end{split}
	\end{equation*}
	This shows \eqref{order_est2} in the case of $n=n_0$. Now, we fix $k > n_0$ and assume that \eqref{order_est2} be satisfied for all $n_0 \leq n < k$. To show \eqref{order_est2} for $n = k$, we let $x \in \Omega_{k}$ and $x' \in \Omega_{k+1}$ with $x_d,x_d'>0$. From \eqref{ur_est}, it follows
	\begin{equation*}
		\begin{split} 
			\frac {\ud}{\ud t} \log \Phi^r(t,x') - \frac {\ud}{\ud t} \log \Phi^r(t,x) \leq \frac 1{(d-1)|B_d|} \int_{Q(\Phi(t,x') \setminus Q(\Phi(t,x))} \frac { |y_h| y_d }{|y|^{d+2}} \omega(t,y) \,\mathrm{d}y + CA.
		\end{split}
	\end{equation*}
	Since \eqref{order_est2} with $n_0 \leq n < k$ are assumed, we can see with \eqref{order_est} that
	\begin{equation*}
		\begin{split} 
			\int_{Q(\Phi(t,x') \setminus Q(\Phi(t,x))} \frac { |y_h| y_d }{|y|^{d+2}} \omega(t,y) \,\mathrm{d}y \leq \int_{\Phi(t,\Omega_{k})} \frac { |y_h| |y_d| }{|y|^{d+2}} |\omega(t,y)| \,\mathrm{d}y .
		\end{split}
	\end{equation*}
	Applying the uniform estimate, we have
	\begin{equation*}
		\begin{split} 
			\frac {\ud}{\ud t} \log \Phi^r(t,x') - \frac {\ud}{\ud t} \log \Phi^r(t,x) \leq C \exp \left( CAt \right) + CA.
		\end{split}
	\end{equation*}
	Therefore, we obtain what we desired. This completes the proof.
\end{proof}

Let $x \in \Omega_k$ for some $k \geq n_0$. Then by the previous lemma, the set $Q(\Phi(t,x))$ can be replaced by the union of bubbles $\cup_{n < k} \Phi(t,\Omega_n)$ on $[0,T]$. From now on, we simply use the notation
\begin{equation*}
	\begin{split}
		I_{n}(t) := \int_{\Phi(t,\Omega_{n})} \frac { |y_h|y_d }{|y|^{d+2}} \omega(t,y) \,\mathrm{d}y.
	\end{split}
\end{equation*}

\begin{lemma}\label{lem_inv}
	Assume that $\omega$ be a solution to \eqref{eq:Euler-axisym-no-swirl} satisfying \eqref{eq:sol-bound} with initial data \eqref{eq:nonexist-data}. Then, there exists a sequence $\{ T_n \}_{n \geq n_0}$ with
	\begin{equation*}
		\begin{split}
			T_n := \min\{ T, c_1(1-\alpha)n^{-1+\alpha} \}
		\end{split}
	\end{equation*}
	for some absolute constant $c_1>0$ such that
	\begin{equation}\label{inv_est}
		\begin{split}
			\frac r2 \leq \Phi^r(t,x) \leq 2r \qquad \mbox{and} \qquad \frac {x_d}2 \leq \Phi^d(t,x) \leq 2x_d
		\end{split}
	\end{equation}
	for all $x \in \Omega_n$ with $n \geq n_0$ and $t \in [0,T_n]$.
\end{lemma}
\begin{proof}
	We prove it inductively. Let $x \in \Omega_{n_0}$ with $x_d>0$ and assume that $\Phi^r(t,x)/\Phi^d(t,x) \leq 64$ for $t \in [0,T_{n_0}]$. We note that this assumption is actually obsolete by the following estimate. From Lemma~\ref{key_lem}, we have
	\begin{equation*}
		\begin{split}
			-CA \leq \frac {\ud}{\ud t} \log \Phi^r(t,x) \leq CA, \qquad -CA \leq \frac {\ud}{\ud t} \log \Phi^d(t,x) \leq CA
		\end{split}
	\end{equation*}
	for all $t \in [0,T]$. Taking $c_0$ small enough to satisfy $CAT \leq Cc_0 \leq \log 2$, we can see that the trajectory $\Phi(t,x)$ cannot attain $\Phi^r(t,x)/\Phi^d(t,x) \geq 64$ at any time $t \in [0,T_{n_0}]$. Thus, we obtain \eqref{inv_est} with $n=n_0$. Now we let $k > n_0$ and suppose \eqref{inv_est} be true for all $n_0 \leq n < k$. By Lemma~\ref{key_lem}, we have
	\begin{equation*}
		\begin{split}
			\left| \frac {\ud}{\ud t} \log \Phi^r(t,x) - \frac 1{(d-1)|B_d|}\sum_{n<k} I_n(t) \right| \leq CA
		\end{split}
	\end{equation*}
	and
	\begin{equation*}
		\begin{split}
			\left| \frac {\ud}{\ud t} \log \Phi^d(t,x) + \frac 1{|B_d|}\sum_{n<k} I_n(t) \right| \leq CA.
		\end{split}
	\end{equation*}
	By the assumption, we clearly have
	\begin{equation*}
		\begin{split}
			I_n(t) \simeq I_n(0) = Cn^{-\alpha}, \qquad t \in [0,T_n]
		\end{split}
	\end{equation*}
	for all $n_0 \leq n < k$. Using it with $T_k \leq T_n$, we have
	\begin{equation*}
		\begin{split}
			- CA \leq \frac {\ud}{\ud t} \log \Phi^r(t,x) \leq C\frac {k^{1-\alpha}}{1-\alpha} + CA
		\end{split}
	\end{equation*}
	and
	\begin{equation*}
		\begin{split}
			- C\frac {k^{1-\alpha}}{1-\alpha} - CA \leq \frac {\ud}{\ud t} \log \Phi^r(t,x) \leq CA
		\end{split}
	\end{equation*}
	for all $t \in [0,T_k]$. By the definition of $T_k$, we can see
	\begin{equation}\label{eq:c1}
		\begin{split}
			\left( C\frac {k^{1-\alpha}}{1-\alpha} + CA \right)T_k \leq C\frac {k^{1-\alpha}}{1-\alpha} \left(c_1(1-\alpha)k^{-1+\alpha} \right) + CAT \leq \log 2
		\end{split}
	\end{equation}
	for some constant $c_1>0$. Thus, we can deduce the claim. This completes the proof.
\end{proof}

We recall \eqref{In_est}
\begin{equation*}
	\begin{split} 
		I_n(t) = \int_{\Omega_{n}} \frac {|\Phi^r(t,x)|^{d-1}|\Phi^d(t,x)|}{|\Phi(t,x)|^{d+2}} \frac {|\omega_0(y)|}{|y_h|^{d-2}} \,\mathrm{d}y
	\end{split}
\end{equation*}
and let $\ell = \ell(T) \geq n_0$ satisfy $c_1(1-\alpha)\ell^{-1+\alpha}\leq T$. Then by the above lemma, we can observe that
\begin{equation*}
	\begin{split}
		\int_{0}^{T_{n}} I_{n}(t) \,\ud t \gtrsim T_{n}I_{n}(0) \gtrsim \frac{1}{n}, \qquad n \geq \ell.
	\end{split}
\end{equation*} Hence, we can find an absolute constant $c_2>0$ and $\ell \geq n_0$ such that \begin{equation}\label{eq:c0}
	\begin{split}
		\sum_{ k = \sqrt{n}}^{n-1} \int_{0}^{T_{k}} I_{k}(t) \,\ud t \ge c \left( \frac{1}{\sqrt{n}} + \cdots + \frac{1}{n-1} \right) \ge \log n^{c_2}, \qquad \sqrt{n} \geq \ell.
	\end{split}
\end{equation}
For the sake of convenience, we let $c_2 < 1/4$.

\section{$L^{\infty}$ illposedness}\label{sec:Linfty}
In this section, we prove Theorem~\ref{thm:Lorenz} in the case $d=3$. Adapting the argument to $d>3$ is straightforward. Here, we consider the $C^{\infty}$--smooth function $\omega_0$ given by \eqref{eq:nonexist-data} with some $m < \infty$. From a direct computation, we can verify that 
\begin{equation}\label{Lorentz_est}
	\begin{split}
		\left\| \frac {\omega_0}r \right\|_{L^{3,q}(\bbR^3)}^q \leq C \sum_{n = n_0}^m n^{-\alpha q} , \qquad q \geq 1.
	\end{split}
\end{equation} Therefore, if $q>1/\alp$, $r^{-1}\omg_0$ is bounded in $L^{3,q}$ uniformly in $m$. 
According to \cite{Danchin}, we have a unique global in time solution to \eqref{eq:Euler-axisym-no-swirl} 
\begin{equation}\label{danchin_class}
	\begin{split}
		\omega \in L^{\infty}_{loc}([0,\infty);L^1 \cap L^{\infty}(\bbR^3)), 
	\end{split}
\end{equation}
with
\begin{equation*}
	\begin{split}
		\Big\| \frac {\omega(t)}r \Big\|_{L^{3,1}(\bbR^3)} = \left\| \frac {\omega_0}r \right\|_{L^{3,1}(\bbR^3)} < \infty, \qquad t>0.
	\end{split}
\end{equation*}
We now state and prove the following proposition, which is Theorem \ref{thm:inf} in the case $d = 3$. 

\begin{proposition}\label{prop:illposed}
	We consider the sequence of $C^\infty$--smooth initial data $\omega_{0}^{(m)}$ given by \eqref{eq:nonexist-data}, where $n_0 < m < \infty$ and $0<\alpha<c_2$, with $c_2>0$ being the constant from \eqref{eq:c0}. Given any $\varepsilon>0$ and $q>1/\alp$, by taking $n_0$ large, we have uniform bounds \begin{equation*}
		\begin{split}
			\nrm{\omg_0^{(m)}}_{L^1 \cap L^\infty(\bbR^3)} \le \varepsilon, \qquad \nrm{r^{-1}\omg_0^{(m)}}_{L^{3,q}(\bbR^3)} \le \varepsilon
		\end{split}
	\end{equation*} for all $m\ge n_{0}$.
	Then, there exists $M=M(n_0) >0$ such that for all $m > M$, the unique global in time solution $\omega^{(m)}$ with initial data $\omg^{(m)}_0$ satisfies
	\begin{equation*}
		\begin{split}
			\sup_{t \in [0,T(m)]} \| \omega^{(m)}(t,\cdot) \|_{L^{\infty}(\bbR^3)} \geq \frac 14 m^{c_2 - \alpha}
		\end{split}
	\end{equation*}
	for $T(m) := c_1(1-\alpha) m^{\frac 12(-1+\alpha)}$, where $c_1$ is from \eqref{eq:c1}.
\end{proposition}
\begin{proof} In the proof, we drop the superscript $(m)$ for simplicity. 
	We set
	\begin{equation*}
		\begin{split}
			A(m) = \frac 14 m^{c_2 - \alpha}
		\end{split}
	\end{equation*}
	and observe that
	\begin{equation*}
		\begin{split}
			AT \leq C m^{c_2-\alpha} m^{\frac 12(-1+\alpha)} = C m^{c_2 - \frac 12 - \frac 12\alpha}.
		\end{split}
	\end{equation*}
	Since we have $c_2 - 1/2 - \alpha/2 < 0$, we can find $M>n_0^2$ so that $AT\leq c_0$ for all $m>M$. We fix $M$ to satisfy the above conditions and assume that there exists $m>M$ with
	\begin{equation*}
		\begin{split}
			\sup_{t \in [0,T]}\| \omega(t) \|_{L^{\infty} (\bbR^3)} \leq A.
		\end{split}
	\end{equation*}
	Since $\omega$ satisfies all conditions for \S \ref{sec_std}, we obtain Lemma~\ref{lem:basic}, Lemma~\ref{lem_order}, and Lemma~\ref{lem_inv} with the constants $c_1$ and $c_2$, regardless of the choice of $m$. We notice that $T_{\sqrt{m}} = T$. Thus, we can have from \eqref{eq:c0} with $n=m=\ell^2$ that
	\begin{equation*}
		\begin{split}
			\sum_{ k = \sqrt{m}}^{m-1} \int_{0}^{T_{k}} I_{k}(t) \,\ud t \ge \log m^{c_2}.
		\end{split}
	\end{equation*}
	Let $x \in \Omega_m$. By the use of \eqref{ur_est}, we infer that
	\begin{equation*}
		\begin{split}
			\log \frac {\Phi^r(T,x)}{r} \ge \frac 1{(d-1)|B_d|} \sum_{k<m} \int_0^{T} I_k(\tau)\,\ud \tau - CAT \ge \log m^{c_2} - CAT.
		\end{split}
	\end{equation*}
	Hence, it follows
	\begin{equation*}
		\begin{split}
			\inf_{x\in\Omega_m} \frac {|\Phi^r(T,x)|}{r} \geq \frac 12 m^{c_2}.
		\end{split}
	\end{equation*}
	Since we clearly have
	\begin{equation*}
		\begin{split}
			\| \omega(T) \|_{L^{\infty}(\bbR^3)} \geq \| \omega_{0} \|_{L^{\infty}(\Omega_m)} \inf_{x\in\Omega_m} \frac {|\Phi^r(T,x)|}{r} = m^{-\alpha} \inf_{x\in\Omega_m} \frac {|\Phi^r(T,x)|}{r},
		\end{split}
	\end{equation*}
	it follows
	\begin{equation*}
		\begin{split}
			\| \omega(T) \|_{L^{\infty}(\bbR^3)} \geq \frac 12 m^{c_2-\alpha}.
		\end{split}
	\end{equation*}
	This makes a contradiction and we complete the proof.
\end{proof}

\section{Nonexistence}\label{non_exist}\label{sec:nonexistence}
In this section, to prove Theorem~\ref{thm:nonexist}, we consider $\omega_{0}$ satisfying \eqref{eq:nonexist-data} with $m=\infty$ and $1/2 < \alpha < 3/4$. We notify that $\alpha$ will be specified later. We first show a proposition about the initial data.
\begin{proposition}\label{prop_initial}
	Let $\omega_0$ satisfy \eqref{eq:nonexist-data} with $m = \infty$ and $1/2 < \alpha < 3/4$. Then, we have
	\begin{equation*}
		\begin{split}
			\| \omega_{0} \|_{H^{\frac d2}\cap L^{\infty}(\bbR^d)} \to 0 \qquad \mbox{as} \qquad n_0 \to \infty.
		\end{split}
	\end{equation*}
\end{proposition}
{\begin{remark}
	This is not simply obtained by the triangle inequality, and we need to use the fact that the ``bubbles'' of $\omg_0$ are orthogonal in $H^{\frac d2}$. Moreover, by the same method, we can prove that the vector field $\omg_{0} \frac{x_h}{|x_h|}$ belongs to $H^{\frac d 2}$ with $\nrm{  \omg_{0} \frac{x_h}{|x_h|} }_{H^{\frac d2}} \to 0$ as $n_{0} \to \infty$. 
	In particular, this implies that the corresponding velocity $\bfu_0$ in $\bbR^d$ belongs to $H^{\frac d 2 + 1}$ and $\nrm{ \bfu_0}_{H^{\frac d2 + 1 }} \to 0$ as $n_{0} \to \infty$. To see this, we recall that the transform $ \omg_{0} \frac{x_h}{|x_h|}  \mapsto \nb\bfu_0 $ is given by a singular integral operator satisfying the $L^2$-bound, which gives \begin{equation*}
		\begin{split}
			\nrm{ \bfu_0}_{\dot{H}^{\frac d2 + 1 }} \le C \nrm{ \omg_{0} \frac{x_h}{|x_h|}  }_{H^{\frac d2}}
		\end{split}
	\end{equation*} for some absolute constant $C>0$. On the other hand, $L^2$-smallness follows from \begin{equation*}
	\begin{split}
		|\bfu_0(x)| \le C_d \int_{ \bbR^d } \frac{1}{|x-y|^{d-1}} |\omg_0(y)|\, \ud y, 
	\end{split}
\end{equation*} which gives with some $0<a=a(d)<1$ \begin{equation*}
\begin{split}
	\nrm{\bfu_0}_{L^{2}} \le C_d  \nrm{\omg_0}_{L^{\infty}}^{a}  \nrm{\omg_0}_{L^{1}}^{1-a} \le C_d \nrm{\omg_0}_{L^{\infty}}
\end{split}
\end{equation*} after using that $\nrm{\omg_0}_{L^1}\le C \nrm{\omg_0}_{L^\infty}$ holds uniformly in $n_0$. 
\end{remark}}
\begin{proof}
	As we clearly see $\| \omega_{0} \|_{L^2\cap L^{\infty}(\bbR^d)} \to 0$  as $n_0 \to \infty$,
 it suffices to show that
	\begin{equation}\label{initial_est}
		\begin{split}
			\| \Lambda^{\frac d2} \omega_{0} \|_{L^2(\bbR^d)} \to 0 \qquad \mbox{as} \qquad n_0 \to \infty.
		\end{split}
	\end{equation}
	We define operators $\Lambda_h$ and $\Lambda_d$ by
	\begin{equation*}
		\begin{split}
			\Lambda_h f := \mathscr{F}^{-1} |\xi_h| \mathscr{F} f, \qquad \Lambda_d f := \mathscr{F}^{-1}|\xi_d| \mathscr{F} f,
		\end{split}
	\end{equation*}
	for $f \in \mathscr{S}(\bbR^d)$. Using the notation
	\begin{equation*}
		\begin{split}
			\mathscr{F}_h f(\xi_h,x_d) := \int_{\bbR^{d-1}} e^{-ix_h \cdot \xi_h} f(x) \,\mathrm{d}x_h, \qquad \mathscr{F}_d f(x_h,\xi_d) := \int_{\bbR} e^{-ix_d \xi_d} f(x) \,\mathrm{d}x_d,
		\end{split}
	\end{equation*}
	we can verify that $\Lambda_h f = \mathscr{F}_h^{-1} |\xi_h| \mathscr{F}_h f$ and $\Lambda_d f := \mathscr{F}_d^{-1}|\xi_d| \mathscr{F}_d f.$
	We shall first show
	\begin{equation}\label{initial_d_est}
		\begin{split}
			\| \Lambda_d^{\frac d2} \omega_0 \|_{L^2(\bbR^d)}^2 \to 0 \qquad \mbox{as} \qquad n_0 \to \infty.
		\end{split}
	\end{equation}
	By the Plancherel theorem, we have
	\begin{equation*}
		\begin{split}
			\| \Lambda_d^{\frac d2} \omega_0 \|_{L^2(\bbR^d)}^2 = \left\| |\xi_d|^{\frac d2} \mathscr{F}_d \left( \sum_{n=n_0}^{\infty}  n^{-\alp} \omega^{(n)}_{0, loc} \right) \right\|_{L^2(\bbR^d)}^2.
		\end{split}
	\end{equation*}
	For each $\bar{x}_h \in \bbR^{d-1}$, we can find an unique natural number $k \geq n_0$ with
	\begin{equation*}
		\begin{split}
			\sum_{n=n_0}^{\infty}  n^{-\alp} \omega^{(n)}_{0, loc}(\bar{x}_h,x_d) = k^{-\alp} \omega^{(k)}_{0, loc}(\bar{x}_h,x_d).
		\end{split}
	\end{equation*}
	This implies
	\begin{equation*}
		\begin{split}
			\left\| |\xi_d|^{\frac d2} \mathscr{F}_d \left( \sum_{n=n_0}^{\infty}  n^{-\alp} \omega^{(n)}_{0, loc} \right) \right\|_{L^2(\bbR^d)}^2 = \sum_{n=n_0}^{\infty} \left\| |\xi_d|^{\frac d2} n^{-\alp} \mathscr{F}_d \omega^{(n)}_{0, loc} \right\|_{L^2(\bbR^d)}^2 = \sum_{n=n_0}^{\infty} n^{-2\alpha} \Big\| \Lambda_d^{\frac d2} \omega^{(n)}_{0, loc} \Big\|_{L^2(\bbR^d)}^2.
		\end{split}
	\end{equation*}
	Using Gagliardo-Nirenberg interpolation inequality, we can have
	\begin{equation*}
		\begin{split}
			\Big\| \Lambda_d^{\frac d2} \omega^{(n)}_{0, loc} \Big\|_{L^2(\bbR^d)}^2 \leq C \left\| \partial_d^d \omega^{(n)}_{0, loc} \right\|_{L^2(\bbR^d)} \left\| \omega^{(n)}_{0, loc} \right\|_{L^2(\bbR^d)} \leq C
		\end{split}
	\end{equation*}
	for all $n \geq n_0$. Thus, we obtain 
	\begin{equation*}
		\begin{split}
			\sum_{n=n_0}^{\infty} n^{-2\alpha} \Big\| \Lambda_d^{\frac d2} \omega^{(n)}_{0, loc} \Big\|_{L^2(\bbR^d)}^2 \leq C n_0^{1-2\alpha}.
		\end{split}
	\end{equation*}
	By the condition $\alpha>1/2$, \eqref{initial_d_est} follows. On the other hand, for each $\bar{x}_d \in \bbR$, we have
	\begin{equation*}
		\begin{split}
			\sum_{n=n_0}^{\infty}  n^{-\alp} \omega^{(n)}_{0, loc}(x_h,\bar{x}_d) = k^{-\alp} \omega^{(k)}_{0, loc}(x_h,\bar{x}_d)
		\end{split}
	\end{equation*}
	for some $k \geq n_0$. Repeating the above procedure with this property, we can similarly show that
	\begin{equation*}
		\begin{split}
			\| \Lambda_h^{\frac d2} \omega_0 \|_{L^2(\bbR^d)}^2 \to 0 \qquad \mbox{as} \qquad n_0 \to \infty.
		\end{split}
	\end{equation*}
	Combining the above estimates with
	\begin{equation*}
		\begin{split}
			\| \Lambda^{\frac d2} \omega_0 \|_{L^2(\bbR^d)}^2 \simeq \| \Lambda_h^{\frac d2} \omega_0 \|_{L^2(\bbR^d)}^2 + \| \Lambda_d^{\frac d2} \omega_0 \|_{L^2(\bbR^d)}^2,
		\end{split}
	\end{equation*}
	we deduce \eqref{initial_est}. This completes the proof.
\end{proof}

Now, we let $A>0$ and $T>0$ satisfy $AT \leq c_0$ and assume that there exists a solution to \eqref{eq:Euler-axisym-no-swirl} satisfying 
\begin{equation*}
	\begin{split}
		\omega \in L^{\infty}([0,T];H^{\frac d2}(\bbR^d)), \qquad 		\sup_{t \in [0,T]} \| \Lambda^{\frac d2} \omega(t) \|_{L^2 (\bbR^d)} \leq A.
	\end{split}
\end{equation*}
Then, following \S \ref{sec_std}, we obtain the statements of Lemmas~\ref{lem:basic}, \ref{lem_order}, and \ref{lem_inv} with the same constants $c_1$ and $c_2$. We may write the solution in the form
\begin{equation*}
	\begin{split}
		\omega(t,x) = \sum_{n=n_0}^{\infty} n^{-\alpha} \omega^{(n)}_{loc}(t,x), \qquad \omega^{(n)}_{loc}(t,\Phi(t,x)) = \omega_{0,loc}^{(n)}(t,x).
	\end{split}
\end{equation*}
Then, we can observe from \eqref{eq:claim} that the set $D_n(t)$ defined by
\begin{equation*}
	\begin{split}
		D_n(t) := \{ x_h \in \bbR^{d-1} ; \omega^{(n)}_{loc}(t,x) \neq 0 \mbox{ for some } x_d \in \bbR \}
	\end{split}
\end{equation*}
is disjoint from each other on the time interval $[0,T]$. Thus, we have that
\begin{equation*}
	\begin{split}
		\{ x_h \in \bbR^{d-1} ; \Lambda_d^{\frac d2} \omega^{(n)}_{loc}(t,x) \neq 0 \mbox{ for some } x_d \in \bbR \} = D_n(t),
	\end{split}
\end{equation*}
directly obtained by performing the Fourier transform with respect to $x_d$. Using this property, we show a friendly lower bound of $\| \Lambda^{\frac d2} \omega(t) \|_{L^2(\bbR^d)}$.
\begin{proposition}
	Assume that $\omega$ be a solution to \eqref{eq:Euler-axisym-no-swirl} satisfying \eqref{eq:sol-bound} with initial data \eqref{eq:nonexist-data}. Then, we have an absolute constant $c>0$ such that
	\begin{equation}\label{lb_est}
		\begin{split}
			\int_{\bbR^d} |\Lambda^{\frac d2} \omega(t)|^2 \,\mathrm{d}x \geq c \sum_{n=n_0}^{\infty} \left( \inf_{x\in\Omega_n} \frac {|\Phi^r(t,x)|}{r} \right)^{2(d-2)} \left( \inf_{x\in\Omega_n} \frac {|x_d|}{|\Phi^d(t,x)|} \right)^d n^{-2\alpha}.
		\end{split}
	\end{equation}
	for all $t \in [0,T]$.
\end{proposition}
\begin{proof}
	The above observation yields
	\begin{equation*}
		\begin{split}
			\int_{\bbR^d} |\Lambda^{\frac d2} \omega(t)|^2 \,\mathrm{d}x \geq \int_{\bbR^d} |\Lambda_d^{\frac d2} \omega(t)|^2 \,\mathrm{d}x = \sum_{n=n_0}^{\infty} \int_{\bbR^d} |\Lambda_d^{\frac d2} \omega^{(n)}_{loc}(t)|^2 \,\mathrm{d}x.
		\end{split}
	\end{equation*}
	By the Gagliardo–Nirenberg interpolation inequality
	\begin{equation*}
		\begin{split}
			\| \partial_d \omega^{(n)}_{loc} \|_{L^2} \leq C \| \Lambda_d^{\frac d2} \omega^{(n)}_{loc} \|_{L^2}^{\frac 2d} \| \omega^{(n)}_{loc} \|_{L^2}^{\frac {d-2}d},
		\end{split}
	\end{equation*}
	we have
	\begin{equation*}
		\begin{split}
			\int_{\bbR^d} |\Lambda^{\frac d2} \omega(t)|^2 \,\mathrm{d}x \geq c \sum_{n=n_0}^{\infty} \left( \int_{\bbR^d} |\partial_d \omega^{(n)}_{loc}(t)|^2 \,\mathrm{d}x \right)^{\frac {d}{2}} \left( \int_{\bbR^d} |\omega^{(n)}_{loc}(t)|^2 \,\mathrm{d}x \right)^{- \frac {d}{2}+1}.
		\end{split}
	\end{equation*}
	We can directly estimate
	\begin{align*}
		\int_{\bbR^d} |\partial_d \omega^{(n)}_{loc}(t)|^2 \,\mathrm{d}x &\geq \int_{\bbR^d} \frac {|\omega^{(n)}_{loc}(t,x)|^2}{|x_d|^2} \,\mathrm{d}x  = \int_{\Omega_n} \frac {|\Phi^r(t,x)|^{2(d-2)}}{|\Phi^d(t,x)|^2} \frac {|\omega_0(x)|^2}{|x_h|^{2(d-2)}} \,\mathrm{d}x \\
		&\geq c \left( \inf_{x\in\Omega_n} \frac {|\Phi^r(t,x)|}{r} \right)^{2(d-2)} \left( \inf_{x\in\Omega_n} \frac {|x_d|}{|\Phi^d(t,x)|} \right)^2 8^{2n} \int_{\Omega_n} |\omega_0|^2 \,\mathrm{d}x
	\end{align*}
	and
	\begin{align*}
		\int_{\bbR^d} |\omega^{(n)}_{loc}(t)|^2 \,\mathrm{d}x 
		&= \int_{\Omega_n} |\Phi^r(t,x)|^{2(d-2)} \frac {|\omega_0(x)|^2}{|x_h|^{2(d-2)}} \,\mathrm{d}x  \leq \left( \sup_{x\in\Omega_n} \frac {|\Phi^r(t,x)|}{r} \right)^{2(d-2)} \int_{\Omega_n} |\omega_0|^2 \,\mathrm{d}x.
	\end{align*}
	Combining the above estimates with
	\begin{equation*}
		\begin{split}
			\inf_{x\in\Omega_n} \frac {|\Phi^r(t,x)|}{r} \simeq \sup_{x\in\Omega_n} \frac {|\Phi^r(t,x)|}{r},
		\end{split}
	\end{equation*}
	a direct consequence of \eqref{eq:claim}, we obtain
	\begin{equation*}
		\begin{split}
			\int_{\bbR^d} |\Lambda^{\frac d2} \omega(t)|^2 \,\mathrm{d}x \geq c \sum_{n=n_0}^{\infty} \left( \inf_{x\in\Omega_n} \frac {|\Phi^r(t,x)|}{r} \right)^{2(d-2)} \left( \inf_{x\in\Omega_n} \frac {|x_d|}{|\Phi^d(t,x)|} \right)^d 8^{dn} \int_{\Omega_n} |\omega_0|^2 \,\mathrm{d}x.
		\end{split}
	\end{equation*}
	Recalling the definition of $\omega_0$, we obtain \eqref{lb_est}. This completes the proof.
\end{proof}

To make $\dot{H}^{\frac d2}$ norm inflation, it is essential to control the position of each bubble in the direction of $x_d$. The following lemma provides a bound. 
\begin{lemma}\label{lem_Phid}
	Assume that $\omega$ be a solution to \eqref{eq:Euler-axisym-no-swirl} satisfying \eqref{eq:sol-bound} with initial data \eqref{eq:nonexist-data}. Then, there exists an absolute constant $C>0$ such that
	\begin{equation}\label{Phid_est}
		\begin{split}
			\log \frac {\Phi^d(t,x)}{x_d} \le -\frac 1{4|B_d|} \sum_{k<n} \int_0^{t} I_k(\tau)\,\ud \tau + CAt
		\end{split}
	\end{equation}
	for all $x \in \Omega_n$ with $n > n_0$ and $t \in [0,T]$.
\end{lemma}
\begin{proof}
	Let $x \in \Omega_n$ with $n > n_0$ and $x_d>0$. From \eqref{ur_est}, we clearly have
	\begin{equation*}
		\begin{split}
			\left| \frac {\ud}{\ud t} \log \Phi^r(t,x) - \frac 1{(d-1)|B_d|} \sum_{k<n} I_k(t) \right| \le CA
		\end{split}
	\end{equation*}
	and
	\begin{equation*}
		\begin{split}
			\log \frac {\Phi^r(t,x)}{r} \le C \sum_{k<n} \int_0^{t} I_k(\tau)\,\ud \tau + CAt.
		\end{split}
	\end{equation*}
	On the other hand, we show from \eqref{ud_est} that
	\begin{equation*}
		\begin{split}
			\left| \frac {\ud}{\ud t} \log \Phi^d(t,x) + \frac 1{|B_d|} \sum_{k<n} I_k(t) \right| \le CA \left(1 + \log \frac {\Phi^r(t,x)}{\Phi^d(t,x)} \right)^{\frac {d-1}d},
		\end{split}
	\end{equation*}
	hence,
	\begin{equation}\label{dt_log_est}
		\begin{split}
			\frac {\ud}{\ud t} \log \Phi^d(t,x) \leq - \frac 1{|B_d|} \sum_{k<n} I_k(t) + CA \left( 1+\log \frac {\Phi^r(t,x)}{r} - \log \frac {\Phi^d(t,x)}{x^d} \right).
		\end{split}
	\end{equation}
	Suppose that $T^* \in (0,T]$ be the first time such that
	\begin{equation*}
		\begin{split}
			\log \frac {\Phi^d(t,x)}{x_d} \ge -\frac 1{2|B_d|} \sum_{k<n} \int_0^{t} I_k(\tau)\,\ud \tau + CAt.
		\end{split}
	\end{equation*}
	Note that we can ensure $T^* \neq 0$ by taking the constant $C$ in the above assumption large. Integrating \eqref{dt_log_est} over time and combining the above, we can have
	\begin{gather*}
		\log \frac {\Phi^d(t,x)}{x_d} = \log \frac {\Phi^d(T^*,x)}{x_d} - \frac 1{|B_d|} \sum_{k<n} \int_{T^*}^t I_k(\tau) \,\ud \tau + CA (t-T^*) \left( 1+\sum_{k<n} \int_0^{t} I_k(\tau)\,\ud \tau \right) \\
		=-\frac 1{2|B_d|} \sum_{k<n} \int_0^{T^*} I_k(\tau)\,\ud \tau + CAT^* - \frac 1{|B_d|} \sum_{k<n} \int_{T^*}^t I_k(\tau) \,\ud \tau + CA (t-T^*) \left( 1+\sum_{k<n} \int_0^{t} I_k(\tau)\,\ud \tau \right) \\
		\leq -\frac 1{2|B_d|} \sum_{k<n} \int_0^{t} I_k(\tau) \,\ud \tau + CA(t-T^*) \sum_{k<n} \int_0^{t} I_k(\tau)\,\ud \tau + CAt
	\end{gather*}
	until the assumption is satisfied. We impose the condition $CAT \leq Cc_0 \leq 1/4|B_d|$. Then, we can infer \eqref{Phid_est}. This completes the proof.
\end{proof}

\begin{proof}[Proof of Theorem~\ref{thm:nonexist}]
Now we are ready to prove Theorem~\ref{thm:nonexist}. For given $\varepsilon > 0$, we take $n_0 \in \bbN$ sufficiently large, so that by Proposition~\ref{prop_initial}, we have
\begin{equation*}
	\begin{split}
		\nrm{\omega_0}_{H^{\frac d2} \cap L^{\infty}(\bbR^d)} < c\varepsilon 
	\end{split}
\end{equation*} for any $c>0$. Then, {as we have remarked after Proposition \ref{prop_initial}, this implies that $\bfu_0 \in H^{\frac d2 + 1}$ with $\nrm{\bfu_0}_{H^{\frac d2 + 1}} < \varepsilon$ by choosing small $c>0$}. 
Let $\ell \geq n_0$ be given in \eqref{eq:c0} and $x \in \Omega_n$ with $\sqrt{n} > \ell$. Then, from \eqref{ur_est} we have
\begin{equation*}
	\begin{split}
		\log \frac {\Phi^r(T,x)}{r} \ge \frac 1{(d-1)|B_d|} \sum_{k<n} \int_0^{T} I_k(\tau)\,\ud \tau - CAT \ge \log n^{c_2} - CAT,
	\end{split}
\end{equation*}
and from Lemma~\ref{lem_Phid},
\begin{equation*}
	\begin{split}
		\log \frac {|x_d|}{|\Phi^d(T,x)|} \ge \frac 1{4|B_d|} \sum_{k<n} \int_0^{T} I_k(t)\,\ud t - CAT \ge \log n^{c_2} - CAT.
	\end{split}
\end{equation*}
Thus, we have
\begin{equation*}
	\begin{split}
		\inf_{x\in\Omega_n} \frac {|\Phi^r(T,x)|}{r} \geq \frac 12 n^{c_2}, \qquad \inf_{x\in\Omega_n} \frac {|x_d|}{|\Phi_d(T,x)|} \geq \frac 12 n^{c_2}.
	\end{split}
\end{equation*}
Applying them to \eqref{lb_est}, we obtain
\begin{equation*}
	\begin{split}
		\int_{\bbR^d} |\Lambda^{\frac d2} \omega(T)|^2 \,\mathrm{d}x \geq c_3 \sum_{n \geq \ell^2} n^{c_4 - 2\alpha}
	\end{split}
\end{equation*}
for some $c_3>0$ and $0<c_4<1/4$. Here, we take $\alpha > 1/2$ with $c_4 - 2\alpha > -1$. This clearly leads to a contradiction. Taking $A$ large and $T$ small enough with $AT \leq c_0$, we can deduce the nonexistence result. This completes the proof.\end{proof}

\section{Norm inflation for smooth data}\label{sec:norm-inflation}
In this section, we establish Theorem~\ref{thm:inf}. We take $\alpha$ as in \S \ref{non_exist} and let $\omega_{0}$ satisfy \eqref{eq:nonexist-data} with $m<\infty$. In this case with $d=3$, we have already shown that there exists a unique global in time solution $\omega$ to \eqref{eq:Euler-axisym-no-swirl} with \eqref{danchin_class}. We believe that we can similarly extend this result to the $d>3$ cases by the use of exterior derivative. We prove Theorem~\ref{thm:inf} by providing the following proposition.

\begin{proposition}[Quantitative norm inflation]
	We consider the $C^\infty$--smooth initial data $\omega_{0}$ given by \eqref{eq:nonexist-data} with $n_0 < m < \infty$ and $1/2<\alpha<(1+c_4)/2$. Then, there exists $M=M(n_0) >0$ such that for all $m > M$, the unique global in time solution $\omega$ with \eqref{danchin_class} satisfies
	\begin{equation*}
		\begin{split}
			\sup_{t \in [0,T(m)]} \int_{\bbR^d} |\Lambda^{\frac d2} \omega(t)|^2 \,\mathrm{d}x \geq \frac {c_3}2 \sum_{n=m/2}^{m} n^{c_4 - 2\alpha}, \quad \mbox{for} \quad T(m) := c_1(1-\alpha)(m/2)^{\frac 12(-1+\alpha)}.
		\end{split}
	\end{equation*} 
\end{proposition}
\begin{proof}
	We set
	\begin{equation*}
		\begin{split}
			A(m) := \frac {c_3}2 \sum_{n=m/2}^{m} n^{c_4 - 2\alpha}
		\end{split}
	\end{equation*}
	and see that
	\begin{equation*}
		\begin{split}
			AT \leq C m^{1+c_4-2\alpha} m^{\frac 12(-1+\alpha)} = C m^{\frac 12 + c_4 - \frac 32\alpha}.
		\end{split}
	\end{equation*}
	By $1/2 + c_4 - 3\alpha/2 < 0$, we can let $M>2n_0^2$ so that $AT \leq c_0$ for all $m>M$. Suppose that with
	\begin{equation*}
		\begin{split}
			\sup_{t \in [0,T]}\| \Lambda^{\frac d2} \omega(t) \|_{ L^2 (\bbR^d)} \leq A
		\end{split}
	\end{equation*}
	for some $m>M$. Then, we can show Lemma~\ref{lem:basic}, Lemma~\ref{lem_order}, Lemma~\ref{lem_inv}, and Lemma~\ref{lem_Phid} with the same constants, regardless of the choice of $m$. Using $T_{k} \leq T$ for all $k \geq \sqrt{m/2}$, we show
	\begin{equation*}
		\begin{split}
			\sum_{ k = \sqrt{n}}^{n-1} \int_{0}^{T_{k}} I_{k}(t) \,\ud t \ge \log n^{c_2}, \qquad m/2 \leq n \leq m
		\end{split}
	\end{equation*}
	as we obtained \eqref{eq:c0}. Since this implies
	\begin{equation*}
		\begin{split}
			\inf_{x\in\Omega_n} \frac {|\Phi^r(T,x)|}{r} \geq \frac 12 n^{c_2}, \qquad \inf_{x\in\Omega_n} \frac {|x_d|}{|\Phi_d(T,x)|} \geq \frac 12 n^{c_2}
		\end{split}
	\end{equation*}
	for all $m/2 \leq n \leq m$, we apply them to \eqref{lb_est} and obtain
	\begin{equation*}
		\begin{split}
			\int_{\bbR^d} |\Lambda^{\frac d2} \omega^{(m)}(T)|^2 \,\mathrm{d}x \geq c_3 \sum_{n=m/2}^{m} n^{c_4 - 2\alpha},
		\end{split}
	\end{equation*}
	which makes a contradiction. 
	This completes the proof.
\end{proof}

\section{Proof of the Key Lemma}\label{sec:key-Lemma}
We only treat the case $d=3$ because the others can be estimated similarly. We first prove \eqref{ur_est}, obtaining
\begin{equation}\label{ur_3_est}
	\begin{split}
		\left| \frac {u^r(x)}{r} - \frac 3{8\pi} \int_{Q(x)} \frac { |y_h| y_3 }{|y|^{5}} \omega(y) \,\mathrm{d}y \right| \le C \| \nabla \omega \|_{L^3(\bbR^3)}
	\end{split}
\end{equation}
and
\begin{equation}\label{ur_3_est_2}
	\begin{split}
		\left| \frac {u^r(x)}{r} - \frac 3{8\pi} \int_{Q(x)} \frac { |y_h| y_3 }{|y|^{5}} \omega(y) \,\mathrm{d}y \right| \le C \| \omega \|_{L^\infty(\bbR^3)}
	\end{split}
\end{equation}
respectively. We recall \eqref{eq:BS-r}
\begin{equation*}
	\begin{split}
		u^r(x) =  -\frac{1}{4\pi|x_h|}\int_{\bbR^3} \frac{ (x_3-y_3) (x_1y_1 + x_2y_2) }{|y_h||x-y|^3}  \omg(y) \, \ud y.
	\end{split}
\end{equation*}
We fix $x \in \bbR^3$ with $|x_h| \geq x_3 > 0$ and define the sets $R \subset \bbR^3$ and $S \subset \bbR^3$ by
\begin{align*}
	R(x) &:= \{ y \in \bbR^3 ; |y_h| \leq 4|x_h|,\, |y_3| \geq 4|x_h| \}, \qquad	S(x) := \{ y \in \bbR^3 ; |y_h| \leq 4|x_h|,\, |y_3| \leq 4|x_h| \}.
\end{align*}
Using $\bbR^3 = Q(x) \cup R(x) \cup S(x)$, we write
\begin{equation}\label{I123_est}
	\begin{split}
		u^r(x) = \frac 1{4\pi} (I_1 + I_2 + I_3),
	\end{split}
\end{equation}
where $I_1, I_2, I_3$ refer to the integral over $Q(x), R(x), S(x)$, respectively. Due to $|y|^2 \geq 8|x|^2$ for $y \in Q(x)$, we can perform the Taylor expansion and obtain
\begin{equation}\label{Taylor_exp}
	\begin{split}
		\frac {1}{|x-y|^3} = \frac 1{|y|^3} \left( 1-\frac {2(x \cdot y)}{|y|^2} + \frac {|x|^2}{|y|^2} \right)^{-\frac 32} = \frac 1{|y|^3} + \frac {3(x\cdot y)}{|y|^5} + O\left( \frac {|x|^2}{|y|^5} \right).
	\end{split}
\end{equation}
Thanks to the following estimate 
\begin{equation*}
	\begin{split}
		&\frac 1{|x_h|} \int_{Q(x)} \frac {|(x_3-y_3)(x_h \cdot y_h)|}{|y_h|} \frac {|x|^2}{|y|^5} |\omega(y)| \,\mathrm{d}y \leq C |x|^2 \int_{Q(x)} \frac {1}{|y|^3} \frac {|\omega(y)|}{|y|} \,\mathrm{d}y \leq C |x| \left\| \frac {\omega(y)}{|y_h|} \right\|_{L^3(Q(x))},
	\end{split}
\end{equation*} 
we can deduce
\begin{equation*}
	\begin{split}
		|I_1 - ( I_{11} + I_{12} + I_{13} )| \leq  C |x| \left\| \frac {\omega(y)}{|y_h|} \right\|_{L^3(\bbR^3)},
	\end{split}
\end{equation*}
where
\begin{align*}
	I_{11} &:= -\frac{1}{|x_h|}\int_{Q(x)} \frac{ (x_3-y_3) (x_h \cdot y_h) }{|y_h||y|^3}  \omg(y) \, \ud y, \\
	I_{12} &:= -\frac{3}{|x_h|}\int_{Q(x)} \frac{ x_3(x_h \cdot y_h)(x \cdot y)}{|y_h||y|^5} \omg(y) \, \ud y, \\
	I_{13} &:= \frac{3}{|x_h|}\int_{Q(x)} \frac{ y_3(x_h \cdot y_h) (x \cdot y)}{|y_h||y|^5} \omg(y) \, \ud y.
\end{align*}
To begin with, the  property 
\begin{equation}\label{cond_r_sym}
	\begin{split}
		\omega(y_h,y_3) = \omega(-y_h,y_3)
	\end{split}
\end{equation}
shows that $I_{11} = 0.$ Next, writing $I_{12}$ as
\begin{equation*}
	\begin{split}
		I_{12} = -\frac{3}{|x_h|}\int_{Q(x)} \frac{ x_3(x_h \cdot y_h)^2}{|y_h||y|^5} \omg(y) \, \ud y -\frac{3}{|x_h|}\int_{Q(x)} \frac{ x_3(x_h \cdot y_h)(x_3y_3)}{|y_h||y|^5} \omg(y) \, \ud y,
	\end{split}
\end{equation*}
we use \eqref{cond_d_sym} and \eqref{cond_r_sym} respectively to get $I_{12} = 0.$
Similarly, we separate $I_{13}$ into two parts
\begin{equation*}
	\begin{split}
		I_{13} = \frac{3}{|x_h|}\int_{Q(x)} \frac{ y_3(x_h \cdot y_h)^2}{|y_h||y|^5} \omg(y) \, \ud y +\frac{3}{|x_h|}\int_{Q(x)} \frac{ y_3(x_h \cdot y_h)(x_3y_3)}{|y_h||y|^5} \omg(y) \, \ud y.
	\end{split}
\end{equation*}
{We clearly have
\begin{equation*}
	\begin{split}
		\frac{3}{|x_h|}\int_{Q(x)} \frac{ y_3(x_h \cdot y_h)(x_3y_3)}{|y_h||y|^5} \omg(y) \, \ud y = 0.
	\end{split}
\end{equation*}
By $(x_h \cdot y_h)^2 = x_1^2y_1^2 + x_2^2y_2^2 + 2x_1y_1x_2y_2$ and the change of variables $(y_1,y_2,y_3) \mapsto (-y_2,y_1,y_3)$, it follows
\begin{equation*}
	\begin{split}
		I_{13} = \frac{3|x_h|}{2}\int_{Q(x)} \frac{|y_h|y_3}{|y|^5} \omg(y) \, \ud y.
	\end{split}
\end{equation*}
Thus, we have with Hardy's inequality that
}
\begin{equation*}
	\begin{split}
		\left| I_1 - \frac{3|x_h|}{2}\int_{Q(x)} \frac{|y_h|y_3}{|y|^5} \omg(y) \, \ud y \right| \leq C |x| \| \nabla \omega \|_{L^3(\bbR^3)}.
	\end{split}
\end{equation*}
For $y \in R(x)$, we can verify that
\begin{equation}\label{R_est}
	\begin{split}
		\frac 14 |y_3| \leq |x-y| \leq 4|y_3|.
	\end{split}
\end{equation}
 {
By H\"{o}lder's inequality and Hardy's inequality, we have
\begin{equation*}
	\begin{split}
		|I_2| \leq C \int_{R(x)} \frac {|\omega(y)|}{|y_3|^2} \,\mathrm{d}y \leq C \left\| \frac {\omega(y)}{|y_h|} \right\|_{L^3(R(x))} \left\| \frac {|y_h|}{|y_3|^2} \right\|_{L^{\frac 32}(R(x))} \leq C |x| \| \nabla \omega \|_{L^3(\bbR^3)}.
	\end{split}
\end{equation*}
}
Using integration by parts, we write $I_3 = I_{31} + I_{32} + I_{33},$
where
\begin{align*}
	I_{31} &:= -\frac{1}{|x_h|}\int_{\{ |y_h| \leq 4|x_h| \} \cap \{y_3 = 4|x_h|\}} \frac{ x_h \cdot y_h }{|y_h||x-y|} \omg(y) \,\mathrm{d}S, \\
	I_{32} &:= \frac{1}{|x_h|}\int_{\{ |y_h| \leq 4|x_h| \} \cap \{y_3 = -4|x_h|\}} \frac{ x_h \cdot y_h }{|y_h||x-y|} \omg(y) \,\mathrm{d}S, \\
	I_{33} &:= \frac{1}{|x_h|}\int_{S(x)} \frac{ x_h \cdot y_h }{|y_h||x-y|} \partial_3\omg(y) \, \ud y.
\end{align*}
From the observation
\begin{equation*}
	\begin{split}
		2|x_h| \leq |x-y| \leq 8|x_h|, \qquad y \in \{ |y_h| \leq 4|x_h| \} \cap \{y_3 = 4|x_h|\},
	\end{split}
\end{equation*}
we can estimate $I_{31}$ as
\begin{equation*}
	\begin{split}
		|I_{31}| \leq \frac {C}{|x_h|} \int_{\{ |y_h| \leq 4|x_h| \} \cap \{y_3 = 4|x_h|\}} |\omega(y)| \,\mathrm{d}S.
	\end{split}
\end{equation*}
We note that the fundamental theorem of calculus and H\"{o}lder's inequality yield
\begin{equation*}
	\begin{split}
		|\omega(y)| \leq \int_0^{4|x_h|} |\partial_3 \omega(y)| \,\mathrm{d}y_3 \leq C|x_h|^{\frac 23} \| \partial_3 \omega(y_h,\cdot) \|_{L^3(0,4|x_h|)}
	\end{split}
\end{equation*}
for almost all $y \in \{ |y_h| \leq 4|x_h| \} \cap \{y_3 = 4|x_h|\}$. Thus, using H\"{o}lder's inequality again, we obtain $|I_{31}| \leq C |x_h| \| \partial_3 \omega \|_{L^3(S(x))}$ and similarly, we can show that $|I_{32}| \leq C |x_h| \| \partial_3 \omega \|_{L^3(S(x))}$. Next, computing $I_{33}$ with H\"{o}lder's inequality
\begin{equation*}
	\begin{split}
		|I_{33}| \leq \int_{S(x)} \frac {|\partial_3 \omega(y)|}{|x-y|} \,\mathrm{d}y \leq C \| \partial_3 \omega \|_{L^3(S(x))} \left( \int_{S(x)} \frac 1{|x-y|^{\frac 32}} \,\mathrm{d}y \right)^{\frac 23} \leq C |x_h| \| \partial_3 \omega \|_{L^3(S(x))},
	\end{split}
\end{equation*}
we deduce that
\begin{equation*}
	\begin{split}
		|I_3| \leq C |x| \| \nabla \omega \|_{L^3(\bbR^3)}.
	\end{split}
\end{equation*}
Inserting the estimates for $I_1$, $I_2$, and $I_3$ into \eqref{I123_est}, we obtain \eqref{ur_3_est}. To show \eqref{ur_3_est_2}, we slightly modify the estimates. By \eqref{Taylor_exp} and the following estimate
 {
\begin{equation*}
	\begin{aligned}
		\frac 1{|x_h|} \int_{Q(x)} \frac {|x_3-y_3| |x_1y_1 + x_2y_2|}{|y_h|} \frac {|x|^2}{|y|^5} |\omega(y)| \,\mathrm{d}y &\leq C |x|^2 \int_{Q(x)} \frac {1}{|y|^4} |\omega(y)| \,\mathrm{d}y \leq C |x| \| \omega \|_{L^\infty(Q(x))},
	\end{aligned}
\end{equation*}
}
we obtain
\begin{equation*}
	\begin{split}
		|I_1 - ( I_{11} + I_{12} + I_{13} )| \leq C |x| \| \omega \|_{L^\infty(Q(x))}.
	\end{split}
\end{equation*}
Recalling the estimate for $I_{11}$, $I_{12}$, and $I_{13}$, we arrive at
\begin{equation*}
	\begin{split}
		\left| I_1 - \frac{3|x_h|}{2}\int_{Q(x)} \frac{|y_h|y_3}{|y|^5} \omg(y) \, \ud y \right| \leq C |x| \| \omega \|_{L^\infty(\bbR^3)}.
	\end{split}
\end{equation*}
 {
By \eqref{R_est} and H\"{o}lder's inequality, we obtain
\begin{equation*}
	\begin{split}
		|I_2| \leq C \int_{R(x)} \frac {|\omega(y)|}{|y_3|^2} \,\mathrm{d}y \leq C \| \omega \|_{L^\infty(R(x))} \int_{R(x)} \frac 1{|y_3|^2} \,\mathrm{d}y \leq C |x| \| \omega \|_{L^\infty(R(x))}.
	\end{split}
\end{equation*}
}
Using H\"{o}lder's inequality yields
\begin{equation*}
	\begin{split}
		|I_3| \leq \| \omega \|_{L^\infty(S(x))} \int_{S(x)} \frac {1}{|x-y|^2} \,\mathrm{d}y \le C \nrm{\omg}_{L^\infty(S(x))} |x_h|\leq C |x| \| \omega \|_{L^\infty(S(x))}.
	\end{split}
\end{equation*}
Thus, combining the above with \eqref{I123_est}, we obtain \eqref{ur_3_est_2}.

\medskip 

Now we aim to show that
\begin{equation}\label{ud_3_est}
	\begin{split}
		\left| \frac {u^3(x)}{x_3} + \frac 3{4\pi} \int_{Q(x)} \frac { |y_h| y_3 }{|y|^{5}} \omega(y) \,\mathrm{d}y \right| \le C \left(1 + \log \frac {r}{x_3} \right)^{\frac 23} \| \nabla \omega \|_{L^3(\bbR^3)}
	\end{split}
\end{equation}
and
\begin{equation}\label{ud_3_est_2}
	\begin{split}
		\left| \frac {u^3(x)}{x_3} + \frac 3{4\pi} \int_{Q(x)} \frac { |y_h| y_3 }{|y|^{5}} \omega(y) \,\mathrm{d}y \right| \le C \left(1 + \log \frac {r}{x_3} \right) \| \omega \|_{L^\infty(\bbR^3)}.
	\end{split}
\end{equation}
From \eqref{eq:BS-d} we have
\begin{equation*}
	\begin{split}
		u^3 (x) = \frac{1}{4\pi} \int_{\bbR^d} \frac{ (x_h-y_h) \cdot y_h }{|y_h||x-y|^3}  \omg(y) \, \ud y.
	\end{split}
\end{equation*}
Using \eqref{cond_d_sym} and the notation $\bar{y} := (y_h,-y_3)$, we write
\begin{equation*}
	\begin{split}
		u^3 (x) = \frac{1}{4\pi} \int_{\{y_3 \geq 0\}} \frac{ (x_h-y_h) \cdot y_h }{|y_h|} \left( \frac {1}{|x-y|^3} - \frac {1}{|x-\bar{y}|^3} \right) \omg(y) \, \ud y.
	\end{split}
\end{equation*}
Let
\begin{equation}\label{J123_est}
	\begin{split}
		u^3(x) = \frac 1{4\pi} (J_1 + J_2 + J_3),
	\end{split}
\end{equation}
where $J_1, J_2$, and $J_3$ refer to the above integral on $Q(x), R(x)$, and $S(x)$, respectively.
Using the formula
\begin{equation}\label{AB_est}
	\begin{split}
		\frac 1{A^3} - \frac 1{B^3} = \frac {(B^2 - A^2)(A^2 + AB + B^2)}{A^3B^3 (A+B)},
	\end{split}
\end{equation}
we have
\begin{equation*}
	\begin{split}
		J_1 = \int_{Q(x) \cap \{y_3 \geq 0\}} \frac {4x_3y_3((x_h - y_h) \cdot y_h)(|x-y|^2 + |x-y||x-\bar{y}| + |x-\bar{y}|^2)}{|y_h||x-y|^3|x-\bar{y}|^3(|x-y|+|x-\bar{y}|)} \omega(y) \,\mathrm{d}y.
	\end{split}
\end{equation*}
The goal is to prove that
\begin{equation}\label{eq:J1-goal}
	\begin{split}
		\left| J_1 - 6x_3 \int_{Q(x) \cap \{y_3 \geq 0 \}} \frac { |y_h| y_3 }{|y|^5} \omega(y) \,\mathrm{d}y \right| \leq C x_3 \| \nabla \omega \|_{L^3(\bbR^3)}.
	\end{split}
\end{equation}
We consider $J_1 = 4x_3(J_{11} + J_{12}),$
where
\begin{align*}
	J_{11} &:= \int_{Q(x) \cap \{y_3 \geq 0\}} \frac {y_3 (x_h \cdot y_h) (|x-y|^2 + |x-y||x-\bar{y}| + |x-\bar{y}|^2)}{|y_h||x-y|^3|x-\bar{y}|^3(|x-y|+|x-\bar{y}|)} \omega(y) \,\mathrm{d}y, \\
	J_{12} &:= -\int_{Q(x) \cap \{y_3 \geq 0\}} \frac {y_3 |y_h| (|x-y|^2 + |x-y||x-\bar{y}| + |x-\bar{y}|^2)}{|x-y|^3|x-\bar{y}|^3(|x-y|+|x-\bar{y}|)} \omega(y) \,\mathrm{d}y.
\end{align*}
For $y \in Q(x)$, we can show that
 {
\begin{equation}\label{eq:equiv-y}
	\begin{split}
		\frac12 |y| \le |x-y| \le 2|y|, \qquad \frac 12|y| \le |x-\bar{y}| \le 2|y|.
	\end{split}
\end{equation}
}
Thus, we have
\begin{equation*}
	\begin{split}
		|J_{11}| \le C|x| \int_{Q(x) \cap \{y_3 \geq 0\}} \frac {1}{|y|^3} \frac{|\omega(y)|}{|y|} \,\mathrm{d}y  \le C \left\| \frac {\omega(y)}{|y_h|} \right\|_{L^3(Q(x))}.
	\end{split}
\end{equation*}
Regarding $J_{12}$, we claim that 
\begin{equation}\label{J12_est}
	\begin{split}
		\left| J_{12} + \frac{3}{2} \int_{Q(x) \cap \{y_3 \geq 0\}} \frac { |y_h| y_3 }{|y|^5} \omega(y) \,\mathrm{d}y \right| \le C \left\| \frac {\omega(y)}{|y_h|} \right\|_{L^3(Q(x))}.
	\end{split}
\end{equation}
We write $J_{12}=J_{121}+J_{122}+J_{123}$ for
\begin{align*}
	J_{121} &:= -\int_{Q(x) \cap \{y_3 \geq 0\}} \frac {y_3 |y_h| |x-y|^2}{|x-y|^3|x-\bar{y}|^3(|x-y|+|x-\bar{y}|)} \omega(y) \,\mathrm{d}y,\\
	J_{122} &:= -\int_{Q(x) \cap \{y_3 \geq 0\}} \frac {y_3 |y_h| |x-y||x-\bar{y}|}{|x-y|^3|x-\bar{y}|^3(|x-y|+|x-\bar{y}|)} \omega(y) \,\mathrm{d}y,\\
	J_{123} &:= -\int_{Q(x) \cap \{y_3 \geq 0\}} \frac {y_3 |y_h| |x-\bar{y}|^2}{|x-y|^3|x-\bar{y}|^3(|x-y|+|x-\bar{y}|)} \omega(y) \,\mathrm{d}y
\end{align*}
and show that \begin{equation*}
	\begin{split}
		\left| J_{12k} + \frac{1}{2} \int_{Q(x) \cap \{y_3 \geq 0\}} \frac { |y_h| y_3 }{|y|^5} \omega(y) \,\mathrm{d}y \right| \le C \left\| \frac {\omega(y)}{|y_h|} \right\|_{L^3(Q(x))}, \qquad k=1,2,3.
	\end{split}
\end{equation*} We only prove the case $k = 1$, since the others can be treated similarly. We clearly have
\begin{gather*}
	J_{121} + \frac{1}{2} \int_{Q(x) \cap \{y_3 \geq 0\}} \frac { |y_h| y_3 }{|y|^5} \omega(y) \,\mathrm{d}y \\
	= - \int_{Q(x) \cap \{y_3 \geq 0\}} |y_h|y_3 \frac{2|y|^5 - |x-y||x-\bar{y}|^3(|x-y|+|x-\bar{y}|)}{2|y|^5|x-y||x-\tilde{y}|^3(|x-y|+|x-\tilde{y}|)}\omega(y) \,\mathrm{d}y.
\end{gather*}
We note
\begin{equation*}
	\begin{split}
		2|y|^5 - |x-y||x-\bar{y}|^3(|x-y|+|x-\bar{y}|) = |y|^5 - |x-y|^2|x-\bar{y}|^3 + |y|^5 - |x-y||x-\bar{y}|^4 
	\end{split}
\end{equation*}
and
\begin{equation*}
	\begin{split}
		|y|^5 - |x-y|^2|x-\bar{y}|^3 = |y|^3(|y|^2 - |x-y|^2) + |x-y|^2(|y|^3-|x-\bar{y}|^3),
	\end{split}
\end{equation*}
\begin{equation*}
	\begin{split}
		|y|^5 - |x-y||x-\bar{y}|^4 = |y|^4(|y| - |x-y|) + |x-y|(|y|^4-|x-\bar{y}|^4).
	\end{split}
\end{equation*} Applying
\begin{equation*}
	\begin{split}
		|y|-|x-y| = \frac{|y|^2 - |x-y|^2}{|y|+|x-y|}, \qquad |y|-|x-\bar{y}| = \frac{|y|^2 - |x-\bar{y}|^2}{|y|+|x-\bar{y}|}
	\end{split}
\end{equation*} to the above, we have
\begin{equation*}
	\begin{split}
		\left|2|y|^5 - |x-y||x-\bar{y}|^3(|x-y|+|x-\bar{y}|)\right| \le C|x||y|^4.
	\end{split}
\end{equation*} This infers \begin{equation*}
	\begin{split}
		\left|J_{121} + \frac{1}{2} \int_{Q(x) \cap \{y_3 \geq 0\}} \frac { |y_h| y_3 }{|y|^5} \omega(y) \,\mathrm{d}y \right| \le C|x| \int_{Q(x) \cap \{y_3 \geq 0\}} \frac {1}{|y|^3} \frac{|\omega(y)|}{|y|} \,\mathrm{d}y \le C \left\| \frac {\omega(y)}{|y_h|} \right\|_{L^3(Q(x))}.
	\end{split}
\end{equation*} 
With Hardy's inequality, we deduce \eqref{J12_est}. Thus, we obtain \eqref{eq:J1-goal}. As estimating $J_1$, we show
\begin{equation*}
	\begin{split}
		J_2 = \int_{R(x) \cap \{y_3 \geq 0\}} \frac {4x_3y_3((x_h - y_h) \cdot y_h)(|x-y|^2 + |x-y||x-\bar{y}| + |x-\bar{y}|^2)}{|y_h||x-y|^3|x-\bar{y}|^3(|x-y|+|x-\bar{y}|)} \omega(y) \,\mathrm{d}y.
	\end{split}
\end{equation*}
We can recall \eqref{R_est} and show
\begin{equation*}
	\begin{split}
		\left| \frac {4x_3y_3((x_h - y_h) \cdot y_h)(|x-y|^2 + |x-y||x-\bar{y}| + |x-\bar{y}|^2)}{|y_h||x-y|^3|x-\bar{y}|^3(|x-y|+|x-\bar{y}|)} \right| \leq C \frac{x_3}{|y_3|^3}.
	\end{split}
\end{equation*} 
 {
Hence, we have with H\"{o}lder's inequality and Hardy's inequality that
\begin{equation*}
	\begin{split}
		|J_2| \leq Cx_3 \int_{R(x) \cap \{y_3 \geq 0\}} \frac {|\omega(y)|}{|y_3|^3} \,\mathrm{d}y \leq Cx_3 \left\| \frac {\omega(y)}{|y_h|} \right\|_{L^3(R(x))} \left\| \frac {|y_h|}{|y_3|^{3}} \right\|_{L^{\frac 32}(R(x))} \leq Cx_3 \| \nabla \omega \|_{L^3(\bbR^3)}.
	\end{split}
\end{equation*} 
}
To estimate $J_3$, we use the formula
\begin{equation*}
	\begin{split}
		(x_h-y_h) \left( \frac {1}{|x-y|^3} - \frac {1}{|x-\bar{y}|^3} \right) = \nabla_h \left( \frac {1}{|x-y|} - \frac {1}{|x-\bar{y}|} \right)
	\end{split}
\end{equation*} 
and have
\begin{equation*}
	\begin{split}
		J_3 = \int_{S(x) \cap \{y_3 \geq 0\}} \nabla_h \left( \frac {1}{|x-y|} - \frac {1}{|x-\bar{y}|} \right) \cdot \frac{y_h }{|y_h|} \omg(y) \, \ud y.
	\end{split}
\end{equation*} 
With an integration by parts, we can write  $J_3= J_{31} + J_{32},$
where
\begin{align*}
	J_{31} &:= \int_{\{|y_h| = 4|x_h|\} \cap \{0 \leq y_3 \leq 4|x_h|\}} \frac {4x_3y_3}{|x-y||x-\bar{y}|(|x-y|+|x-\bar{y}|)} \omg(y) \,\mathrm{d}S, \\
	J_{32} &:= - \int_{S(x) \cap \{y_3 \geq 0\}} \frac {4x_3y_3}{|x-y||x-\bar{y}|(|x-y|+|x-\bar{y}|)} \nabla_h \cdot \left( \frac{y_h }{|y_h|} \omg(y) \right) \mathrm{d}y.
\end{align*}
From
\begin{equation*}
	\begin{split}
		|x_h| \leq |x-y| \leq |x-\bar{y}|, \qquad y \in \{|y_h| = 4|x_h|\} \cap \{0 \leq y_3 \leq 4|x_h|\},
	\end{split}
\end{equation*} 
we can have
\begin{equation*}
	\begin{split}
		|J_{31}| \leq C \frac {x_3}{|x_h|} \int_0 ^{4|x_h|} |\omega(y)| \,\mathrm{d}y_3, \qquad |y_h|=4|x_h|.
	\end{split}
\end{equation*} 
The fundamental theorem of calculus and H\"{o}lder's inequality imply
\begin{equation*}
	\begin{split}
		|\omega(y)| \leq \int_0^{4|x_h|} |\partial_r \omega(y)| \,\mathrm{d}r \leq C|x_h|^{\frac 13} \| \nabla \omega(\cdot,y_3) \|_{L^3(\{y_h \in \bbR^2;|y_h| \leq 4|x_h|\})}
	\end{split}
\end{equation*} 
for almost all $y \in \{|y_h| = 4|x_h|\} \cap \{0 \leq y_3 \leq 4|x_h|\}$. Thus, it follows
\begin{equation*}
	\begin{split}
		|J_{31}| \leq C x_3 \| \nabla \omega \|_{L^3(S(x))}.
	\end{split}
\end{equation*} 
Using H\"{o}lder's inequality and Hardy's inequality, we have
\begin{equation*}
	\begin{split}
		| J_{32} | \leq Cx_3 \| \nabla \omg \|_{L^3(\bbR^3)} \left( \int_{S(x) \cap \{y_3 \geq 0 \}} \left| \frac {y_3}{|x-y||x-\bar{y}|^2} \right|^{\frac 32} \,\mathrm{d}y \right)^{\frac 23}.
	\end{split}
\end{equation*} 
Here, we note that
\begin{equation}\label{log_est}
	\begin{split}
		\int_{S(x) \cap \{y_3 \geq 0 \}} \left| \frac {y_3}{|x-y||x-\bar{y}|^2} \right|^{\frac 32} \,\mathrm{d}y \leq \left(1 + \log \frac {|x_h|}{x_3} \right),
	\end{split}
\end{equation}
which will be shown at the end of the proof. Hence, combining the above estimates, we obtain
\begin{equation*}
	\begin{split}
		| J_{32} | \leq Cx_3 \| \nabla \omega \|_{L^3(\bbR^3)} \left(1 + \log \frac {|x_h|}{x_3} \right)^{\frac 23}
	\end{split}
\end{equation*} 
and then 
\begin{equation*}
	\begin{split}
		| J_{3} | \leq Cx_3 \| \nabla \omega \|_{L^3(\bbR^3)} \left(1 + \log \frac {|x_h|}{x_3} \right)^{\frac 23}.
	\end{split}
\end{equation*} 
Inserting the estimates for $J_1$, $J_2$, and $J_3$ into \eqref{J123_est}, we obtain \eqref{ud_3_est}. Now, we show \eqref{ud_3_est_2}. We can have from \eqref{eq:equiv-y} that
\begin{equation*}
	\begin{split}
		|J_{11}| \le C|x| \int_{Q(x) \cap \{y_3 \geq 0\}} \frac {1}{|y|^4} |\omega(y)| \,\mathrm{d}y \leq C \| \omega \|_{L^\infty(Q(x))}.
	\end{split}
\end{equation*}
In the above, we estimated for $J_{12} = J_{121} + J_{122} + J_{123}$ that
\begin{equation*}
	\begin{split}
		\left|J_{12k} + \frac{1}{2} \int_{Q(x) \cap \{y_3 \geq 0\}} \frac { |y_h| y_3 }{|y|^5} \omega(y) \,\mathrm{d}y \right| \le C|x| \int_{Q(x) \cap \{y_3 \geq 0\}} \frac {1}{|y|^4} |\omega(y)| \,\mathrm{d}y, \qquad k=1,2,3.
	\end{split}
\end{equation*} 
Since we have already showed that
\begin{equation*}
	\begin{split}
		\int_{Q(x)} \frac {1}{|y|^4} |\omega(y)| \,\mathrm{d}y \leq C |x|^{-1} \| \omega \|_{L^{\infty}(Q(x))},
	\end{split}
\end{equation*}
we can deduce
\begin{equation*}
	\begin{split}
		\left| J_{12} + \frac{3}{2} \int_{Q(x) \cap \{y_3 \geq 0\}} \frac { |y_h| y_3 }{|y|^5} \omega(y) \,\mathrm{d}y \right| \leq C \| \omega \|_{L^\infty(Q(x))}
	\end{split}
\end{equation*}
and
\begin{equation*}
	\begin{split}
		\left| J_{1} + 6x_3 \int_{Q(x) \cap \{y_3 \geq 0\}} \frac { |y_h| y_3 }{|y|^5} \omega(y) \,\mathrm{d}y \right| \leq Cx_3 \| \omega \|_{L^\infty(\bbR^3)}.
	\end{split}
\end{equation*}
Recalling \eqref{R_est}, we have
\begin{equation*}
	\begin{split}
		|J_2| \leq Cx_3 \int_{R(x)} \frac {|\omega(y)|}{|y_3|^3} \,\mathrm{d}y \leq Cx_3 \| \omega \|_{L^\infty(R(x))} \int_{R(x)} \frac {1}{|y_3|^{3}} \,\mathrm{d}y \leq Cx_3 \| \omega \|_{L^\infty(\bbR^3)}.
	\end{split}
\end{equation*}
Using \eqref{AB_est} gives
\begin{equation*}
	\begin{split}
		J_3 = \int_{S(x) \cap \{y_3 \geq 0\}} \frac {4x_3y_3((x_h - y_h) \cdot y_h)(|x-y|^2 + |x-y||x-\bar{y}| + |x-\bar{y}|^2)}{|y_h||x-y|^3|x-\bar{y}|^3(|x-y|+|x-\bar{y}|)} \omega(y) \,\mathrm{d}y.
	\end{split}
\end{equation*}
By H\"{o}lder's inequality, it follows
\begin{equation*}
	\begin{split}
		|J_3| \leq C x_3 \| \omega \|_{L^\infty(S(x))} \int_{S(x) \cap \{y_3 \geq 0\}} \frac {y_3}{|x-y|^2|x-\bar{y}|^2} \,\mathrm{d}y.
	\end{split}
\end{equation*}
Here, we claim that
\begin{equation*}
	\begin{split}
		\int_{S(x) \cap \{y_3 \geq 0 \}} \frac {y_3}{|x-y|^{2}|x-\bar{y}|^{2}} \,\mathrm{d}y \leq C \left(1 + \log \frac {|x_h|}{x_3} \right),
	\end{split}
\end{equation*} 
which actually implies \eqref{log_est}. We consider the integral on the region $S(x) \cap \{0 \leq y_3 \leq 2x_3 \}$ first. Since we can estimate
\begin{align*}
	\int_{S(x) \cap \{0 \leq y_3 \leq 2x_3 \}} \frac {y_3}{|x-y|^{2}|x-\bar{y}|^{2}} \,\mathrm{d}y &\leq \frac 1{x_3} \int_{S(x) \cap \{0 \leq y_3 \leq 2x_3 \}} \frac {1}{|x-y|^{2}} \,\mathrm{d}y \\
	&\leq \frac 2{x_3} \int_0^{x_3} \int_0^{8|x_h|} \frac {r}{r^{2} + y_3^2} \,\mathrm{d}r \mathrm{d}y_3
\end{align*}
and
\begin{equation*}
	\begin{split}
		\int_0^{8|x_h|} \frac {r}{r^{2} + y_3^2} \,\mathrm{d}r = \frac 12 \log \frac {8|x_h|^2+y_3^2}{y_3^2}  \leq \log (C|x_h|) - \log y_3,
	\end{split}
\end{equation*} 
it follows
\begin{equation*}
	\begin{split}
		\frac 2{x_3} \int_0^{x_3} \int_0^{8|x_h|} \frac {r}{r^{2} + y_3^2} \,\mathrm{d}r \mathrm{d}y_3 \leq C \log(C|x_h|) - \frac C{x_3} \int_0^{x_3} \log y_3\, \mathrm{d}y_3 \leq C \left(1 + \log \frac {|x_h|}{x_3} \right).
	\end{split}
\end{equation*} 
Meanwhile, on the region $S(x) \cap \{y_3 \geq 2x_3 \}$, we can see
\begin{equation*}
	\begin{split}
		\int_{S(x) \cap \{y_3 \geq 2x_3 \}} \frac {y_3}{|x-y|^{2}|x-\bar{y}|^{2}} \,\mathrm{d}y \leq C \int_0^{8|x_h|} \int_{x_3}^{\infty} \frac {ry_3}{(r^{2} + y_3^2)^{2}} \,\mathrm{d}y_3\mathrm{d}r.
	\end{split}
\end{equation*} 
 {
Since
\begin{equation*}
	\begin{split}
		\int_0^{8|x_h|} \int_{x_3}^{\infty} \frac {ry_3}{(r^{2} + y_3^2)^{2}} \,\mathrm{d}y_3\mathrm{d}r \leq C \int_0^{8|x_h|} \frac r{r^2+x_3^2} \,\mathrm{d}r \leq C \log \frac {C|x_h|^2 + x_3^2}{x_3^2},
	\end{split}
\end{equation*} 
the claim follows. 
}
Using this estimate, we have
\begin{equation*}
	\begin{split}
		|J_3| \leq C x_3 \| \omega \|_{L^\infty(\bbR^3)} \left(1 + \log \frac {|x_h|}{x_3} \right).
	\end{split}
\end{equation*}
Collecting the estimates for $J_1$, $J_2$ and $J_3$, we obtain \eqref{ud_3_est_2}. This completes the proof.

\subsection*{Acknowledgments}
I.-J. Jeong has been supported by the Samsung Science and Technology Foundation under Project Number SSTF-BA2002-04 and the New Faculty Startup Fund from Seoul National University. J. Kim was supported by a KIAS Individual Grant (MG086501) at Korea Institute for Advanced Study. The authors thank the anonymous referee for a careful reading of the manuscript and giving helpful comments. 

\subsection*{Conflict of interest}

The authors state that there is no conflict of interest.


\bibliographystyle{amsplain}

\end{document}